\documentclass[a4paper,11pt]{article}
\usepackage[utf8]{inputenc}

\usepackage{a4wide} 
\usepackage{amsmath,amssymb,amsthm}
\usepackage{amsfonts, 
mathtools,enumerate,enumitem,esint}
\usepackage{color}
\usepackage{lmodern} 
\usepackage{verbatim} 
\usepackage{hyperref}
\usepackage[normalem]{ulem}

\theoremstyle{plain}
\newtheorem{theorem}{Theorem}[section]
\newtheorem{proposition}[theorem]{Proposition}
\newtheorem{lemma}[theorem]{Lemma}
\newtheorem{corollary}[theorem]{Corollary}
\newtheorem{definition}[theorem]{Definition}

\newtheorem{remark}[theorem]{Remark}
\numberwithin{equation}{section}

\def\bN{\mathbb{N}}
\def\bR{\mathbb{R}}

\def\cH{\mathcal{H}}

\def\cE{\mathcal{E}}
\def\cI{\mathcal{I}}
\def\cJ{\mathcal{J}}
\def\cL{\mathcal{L}}
\DeclareMathOperator{\dv}{div}
\DeclareMathOperator{\sgn}{sgn}
\DeclareMathOperator{\dist}{dist}
\DeclareMathOperator{\supp}{supp}
\def\eps{\varepsilon}
\def\vfi{\varphi}
\def\dd{\mathrm{d}}

\definecolor{wineRed}{rgb}{0.7,0,0.3}
\definecolor{DGreen}{rgb}{0,0.6,0}
\newcommand{\PR}[1]{{\color{wineRed}#1}}
\newcommand{\ml}[1]{{\color{blue}#1}}

\title{Mosco convergence framework for singular limits of~gradient~flows on Hilbert spaces with applications
}

\author{Yoshikazu Giga$^*$, Micha{\l} {\L}asica$^{\dagger}$, Piotr Rybka$^\ddagger$ \\
\small{$^*$ The University of Tokyo, Graduate School of Mathematical Sciences} \\
\small{$^\dagger$ Institute of Mathematics of the Polish Academy of Sciences} \\
\small{$^\ddagger$ University of Warsaw, Faculty  Mathematics, Informatics and Mechanics}}

\begin{document} 

\maketitle

\begin{abstract}
We consider the question of convergence of a sequence of gradient flows defined on different Hilbert spaces. In order to give meaning to this idea, we introduce a notion of connecting operators. This permits us to generalize the concept of Mosco convergence of functionals to our present setting, and state a desired convergence result for gradient flows, which we then prove. We present a variety of examples, including thin domains, dynamic boundary conditions, and discrete-to-continuum limits.

\end{abstract}

\bigskip\noindent
{\bf Key words:} \quad boundary layer, dynamic boundary conditions, convergence of gradient flows, generalized Mosco convergence, differential equations on graphs, total variation flow.

\bigskip\noindent
{\bf 2020 Mathematics Subject Classification.} Primary: 35K20, Secondary: 35B25, 49J45

\section{Introduction} 
Gradient flows, also known as steepest descent flows, are a common type of dynamics, where the velocity of the state variable is given by the negative gradient of a given function on the state space, often referred to as the energy. This notion has been generalized in many directions, see e.g.\ \cite{Brezis}, \cite{AGS05}, \cite{robinson}. Here, we work in the setting of convex analysis on a Hilbert space $X$, where the trajectories of the flow are given by 
\begin{equation}\label{ir0}
u_t \in - \partial \cE(u),\qquad u(0) = u_0,
\end{equation}
where $\cE \colon X \to [0, \infty]$ is convex and lower semicontinuous. The symbol $\partial \cE$ denotes the subdifferential, which is a multivalued operator in general. Among examples of equations of form \eqref{ir0}, there are many parabolic PDEs, such as the heat equation, or the total variation flow.

In asymptotic analysis, one is often led to consider families of gradient flows on $X$ depending on a small parameter $\eps >0$, 
\begin{equation}\label{ir1}
u^\eps_t \in - \partial \cE_\eps(u^\eps),\qquad u^\eps(0) = u^\eps_0.
\end{equation}
The variational structure of \eqref{ir1} allows one to understand the limiting behavior of \eqref{ir1} without identifying $\partial \cE_\eps$; or, if \eqref{ir1} corresponds to a known parabolic PDE, without carrying out any analysis of that PDE. It is natural to ask the following. Suppose the initial conditions $u^\eps_0$ converge to $u_0$. What kind of convergence of the convex functionals $\cE_\eps$ guarantees convergence of the solutions $u^\eps$? This question has been investigated for quite some time. In case $\partial \cE_\eps$ are linear operators, the answer can be given by saying that the resolvents of  $\partial \cE_\eps$ must converge, see e.g.\ \cite[Theorem VIII.20]{reed}. Its generalization to the nonlinear setting led to Mosco convergence, see \cite{mosco},
\cite{watanabe}, \cite[Chapter 3]{attouch}.

In the present paper, we are interested in the question of convergence of solutions to (\ref{ir1}), when the functionals $\cE_\eps$ are no longer defined on the same Hilbert space $X$, but rather each $\cE_\eps$ lives on its own $X_\eps$ and these $X_\eps$'s differ from $X_0$, the space where the limit functional is defined (even though, naturally, they will often be isomorphic). We are particularly interested in the case when $X_\eps = X_\eps(\Omega)$ are function spaces on more or less the same base space $\Omega$, with the topology in $X_\eps$ depending on $\eps$ and $X_0= X_0(\Omega')$, where $\Omega$ and $\Omega'$ may differ significantly, e.g.\ they may have different dimensions. 

We note that in the context of stationary minimization problems in the calculus of variations, such situation does not pose a significant conceptual difficulty. Indeed, one usually embeds $X_\eps$ into a common space (by a change of variables or otherwise), and works directly with the topology of that space. However, in the gradient flow case, the evolution itself is defined in terms of the metric, so one needs to take more care. Naturally, many problems of the form we consider were treated before on a case-by-case basis (see below). Our aim here is to provide a simple and general unified theory. 

In order to be able to compare $u^\eps$ to $u$, we assume that we are given bounded linear maps  $L_\eps \colon X_\eps \to X_0$ that we call \emph{connecting operators}. 
This allows us to discuss strong and weak convergence of sequences $(w_\eps)_{\eps >0}$, $w_\eps \in X_\eps$ along $L_\eps$ (Definition \ref{def:conv_s_W}). We then introduce a concept that we call \emph{Mosco convergence of spaces $X_\eps$ along} $L_\eps$, which asserts that $L_\eps$ are both asymptotically contractive and asymptotically surjective (although they need not be immersions), see Definition \ref{def:connect}. 
Next, we define the concept of \emph{Mosco convergence of functionals} $\cE_\eps$  along $L_\eps$, see Definition \ref{mos-con}. 
Our main result, Theorem \ref{thm:Mosco}, says that if $X_\eps$ Mosco-converge to $X_0$ along $L_\eps$, and $\cE_\eps$ Mosco-converge to $\cE_0$ along $L_\eps$, then the corresponding  gradient flows converge uniformly along $L_\eps$ (see Definition \ref{con-flo}) to the gradient flow of the limit functional. 

After this work was completed, a closely related preprint \cite{MvG} by S.\ Mercer and Y.\ van Gennip appeared, where a much more general framework for convergence of gradient flows defined on different spaces was proposed. Our approach seems much simpler than theirs. A significant difference between the two is also that in \cite{MvG}, compactness (or equicoercivity of the sequence of energies) has to be assumed a priori, while in our case strong convergence of the flows follows from a Mosco-type argument. We also mention the works \cite{KS1, KS2, Kuroda}, where related ideas were investigated in the setting of Dirichlet forms on metric spaces. 

In order to showcase versatility of Theorem \ref{thm:Mosco}, we will present an array of diverse examples where our convergence result applies:
\begin{itemize} 
\item derivation of equations on thin domains,
\item derivation of dynamic boundary conditions for the $p$-heat equation and the total variation flow from boundary layer,
\item showing that equations with dynamic boundary conditions are intermediate between Neumann and Dirichlet data for the total variation flow and the $p$-heat equation,
\item convergence of differential equations on
combinatorial graphs to PDEs on $\bR^n$.
\end{itemize} 

These examples are a significant part of our paper, for that reason let us briefly discuss them. We first offer a new viewpoint on an old problem of deriving equations on thin domains from their counterparts on thick regions.
This problem was first studied by \cite{HaRau} for reaction diffusion equations on a domain of variable thickness whose thin domain limit is a flat space.
 It is extended to various settings for example, for a curved thin domain and a domain with hole \cite{PrizziRR}, \cite{PrizziR}.
 When the limit domain is one-dimensional, the limit equation was already used to construct a stable stationary solution in a thick domain \cite{Ya} by constructing sub and supersolutions.
 The work \cite{HaRau} was first extended to the Navier--Stokes equations by \cite{RauSell} for a flat thin domain.
 Recently, it has been extended for curved thin domain in a series of papers by \cite{Miu1, Miu2, Miu3}.
 This extension is non-trivial since it is a problem for tangent vector fields and not for scalar functions.
 There is a nice review article \cite{raugel} for developments before 1995.
 All these works (except \cite{Ya}) more or less used variational structure of the problems.
 To handle more general equations with no variational structure, recently it has been addressed within the framework of viscosity solutions for elliptic problems, see \cite{isabeau1}, \cite{isabeau2}. Here, we address this problem from the perspective of our abstract theory in the case of the $p$-harmonic flow on a domain whose thickness tends to zero. The details are presented in Section \ref{Sthin}. 

We also study a standard example of the $p$-heat equation, including the limit case $p=1$, i.e., the total variation flow in a region with smooth boundary in $\bR^n$. Since the total variation flow requires an additional treatment in comparison with the $p$-heat equation for $p>1$ we consider these problems in separate sections. We first deal with the case $p>1$.

We are interested in two problems for the $p$-heat flows. The first one is a derivation of the dynamic boundary condition from the boundary layer problem, this is the content of Section \ref{sec5.1}. In the case $p=2$, it is referred to as the concentrating capacity problem, and the dynamic boundary condition is derived as a limit in \cite{ColliR}, \cite{glr}. In this case, the spaces $X_\eps$ are $L^2(\Omega, \|\cdot \|_\eps)$, where the inner product is related to the thickness of the boundary layer. Moreover, $X_0 = L^2(\Omega)\times L^2(\partial\Omega)$ with the standard inner product. This is an illustration that our approach may be applicable to nonlinear problems. We also see that the dimension of the base space changes as $\eps$ goes to zero. In Subsection \ref{sec5.2}, we derive the Neumann and Dirichlet boundary conditions from the dynamic boundary conditions depending on a parameter $\tau$. Thus, we might claim that the dynamic boundary conditions are intermediate between  the Dirichlet and the Neumann data.  In both cases, $X_\eps$ strictly contain $X_0$. We admit general domains $\Omega$ with smooth boundary. The connecting operators are constructed by a directional average over the boundary layer.

In order to complete the treatment of the $p$-heat equation, we consider in Section \ref{S-TV} the case $p=1$, i.e., the total variation flow. We choose to present it separately, because the results are qualitatively different, which is related to the lack of continuity of the boundary trace operator on $BV(\Omega)$. Moreover, dealing with $BV$ functions involves additional technical difficulties, calling for an approximation argument. We note that tackling the problems involving dynamic boundary conditions in $p=1$ was our original motivation for the development of the abstract theory based on variational convergence of energy, which can provide convergence results without the need to carry out any analysis of PDEs, as the system (\ref{tvfdbcbulk}--\ref{tvfdbczbdry}) is quite complicated, and the total variation flow is notorious for its limited regularity.

We close the paper with an example of different nature. Namely, we present the case of diffusion in a discrete setting of a graph with vertices
$V_\varepsilon := \left(\varepsilon \mathbb{Z}/\mathbb{Z}\right)^n \simeq \{ 0, \varepsilon, 2\varepsilon, \ldots, 1-\varepsilon \}^n$.
We show that the gradient flow of the integral of the discrete gradient converges to the gradient flow of the orthotropic Dirichlet integral. In the case of orthotropic total variation ($p$=1), this follows from the fact that the trajectories of the discrete flows are in fact also trajectories of the limiting flow \cite{tetris, Lphd, KSS} and the contractivity of the flow. A similar convergence result (with a convergence rate) has been established by \cite{GGO} for the Allen--Cahn equation in the one-dimensional setting with a different approach.
 In our current study, the main difficulty is constructing the connecting operator. We stress that the primary object here is the equation on a graph, which is not a discretization of a PDE. This opens the door for studying limits of equation on graphs more complicated than $V_\eps$. A similar approach to the study of the discretization of reaction-diffusion systems has been recently presented in \cite{HMS}. Their approach is based on continuous and discrete gradient systems in the continuity equation format. This approach was initiated in \cite{PeS}.


Our approach to the convergence of solutions to (\ref{ir1}) based on a generalization of the Mosco convergence is just one of the possible approaches. We should mention at this point the method based on $\Gamma$-convergence of functionals $\cE_\eps$ combined with the energy-dissipation balance, see \cite{sylvia}. Later, this approach was further developed by Mielke and his collaborators among others, see \cite{M} and more recently \cite{HMS}.  We also contributed to this approach in \cite{glr} by  deriving the dynamic boundary conditions for the heat equation.

We must stress that the notion of gradient flow is quite broad. One important generalization is in the context of the flows in metric spaces as presented in \cite{AGS05}. Recently, the flows in the Wasserstein metric spaces gains more attention due to their application in Machine Learning. An overview of research devoted to these flows may be found in \cite{santa}.

There are papers devoted to the question of stability of the flows in metric spaces with respect to the space, see \cite{fleissner} or \cite{gigli}. We note that the author of \cite{gigli} deals with the Hausdorff--Gromov distance of compact metric spaces, which is different from our setting.

One may say that the approach to gradient flows presented in \cite{robinson} is at the other end of the spectrum, as the role of the functional generating the flow is secondary. What matters is the structure of the global attractor and the orbits connecting critical points. It is not surprising that this approach is quite far from the content of the present paper.

Here is the organization of the paper. Section \ref{sec2} presents the abstract framework and the main result. In Section \ref{s3} we show
the proof of Theorem \ref{thm:Mosco}.
Sections \ref{Sthin} to \ref{SgraphPDE} 
are devoted to the specific examples.

\section{An abstract framework for convergence of gradient flows} \label{sec2}

The concept of Mosco-convergence of a sequence of functionals $\cE_\eps$ defined on a fixed function space is well known, see \cite{attouch}. We present here its extension to the case where not only $\cE_\eps$ depend on $\eps$, but also their space of definition changes as $\eps\to 0^+$. Let $X_0$, $X_\eps$, $\eps >0$ be (real) Hilbert spaces. We first need to introduce a suitable notion of connection between the spaces $X_\eps$ and $X_0$ that will allow us to discuss convergence of sequences $(x_\eps)$, $x_\eps \in X_\eps$. We will simply assume that we are given a family of bounded linear operators $L_\eps \colon X_\eps \to X_0$, $\eps >0$. We will call them \emph{connecting operators}. We will also say that $X_\eps$ are connected to $X_0$ by $L_\eps$. 

\begin{definition} \label{def:conv_s_W}
Let $w \in X_0$, $w^\eps \in X_\eps$, $\eps >0$. We say that $w^\eps$ converge (strongly) to $w$ along $L_\eps$ if 
	\begin{equation} \label{H2_conv}
 \|w^\eps\|_{X_{\eps}} \to \|w\|_{X_0}\quad \text{and} \quad L_{\eps} w^{\eps} \to w \quad \text{in } X_0.
    \end{equation}
 We say that $w^\eps$ converge weakly to $w$ along $L_\eps$ if \begin{equation} \label{H2_conv_w}
 \limsup_{\eps \to 0^+} \|w^\eps\|_{X_{\eps}} < \infty \quad \text{and} \quad L_{\eps} w^{\eps} \rightharpoonup w \quad \text{in } X_0.
    \end{equation} 
\end{definition} 
We stress that the notion of convergence along $L_\eps$ applies to a sequence $(w^\eps)$ and element $w$, a priori all belonging to different spaces $X_\eps$, $X_0$. It is a qualitatively different concept from the usual convergence of sequences in a fixed topological (vector) space. However, it has many similar properties. First of all, uniqueness of the limit follows from uniqueness of the strong/weak limit of the sequence $L_\eps w^\eps$ in $X_0$.

In order to develop further theory, we will require the following additional hypotheses on $L_\eps$. 

\begin{definition} \label{def:connect}
We say that a family $(L_\eps)_{\eps >0}$ of bounded linear operators $L_\eps \colon X_\eps \to X_0$ is \emph{asymptotically contractive} if 
\begin{enumerate}[label={\rm (H\arabic*)}] 
\item \label{H1}    $\limsup_{\eps \to 0} \|L_\eps\|_{B(X_\eps,X_0)} \leq 1.$
\end{enumerate} 
We say that $L_\eps$ are \emph{asymptotically surjective} if  
\begin{enumerate}[label={\rm (H\arabic*)}]
\setcounter{enumi}{1}
\item \label{H2} for every $w \in X_0$, there exists $(w^\eps)$, $w^\eps \in X_\eps$ such that 
$L_\eps w^\eps \rightharpoonup w$ in $X_0$.
\end{enumerate} 
If both \ref{H1} and \ref{H2} hold, we will say that \emph{$X_{\eps}$ Mosco-converge to $X_0$ along $L_\eps$}. 
\end{definition} 
\noindent
Note that \ref{H1} implies that whenever $(w^\eps)$, $w^\eps \in X_\eps$ is \emph{asymptotically bounded}, i.\,e., such that 
\begin{equation} \label{bdd_seq}
\limsup_{\eps \to 0^+} \|w^\eps\|_{X_{\eps}} < \infty,
\end{equation} 
there exists $w \in X_0$ and a subsequence $(w^{\eps_k})$ of $(w^\eps)$ such that
\begin{equation} \label{weak_conv_L}
L_{\eps_k} w^{\eps_k} \rightharpoonup w \quad \text{in } X_0.
\end{equation} 
In this sense, property \ref{H2} is complementary to \ref{H1}. 

The following proposition shows that under the assumption \ref{H1}, strong and weak convergence along $L_\eps$ behave a lot like usual strong and weak convergence in a Hilbert space. 

\begin{proposition} \label{L_conv_prop} Suppose that $L_\eps$ are asymptotically contractive. Then:
\begin{itemize} 
\item[(i)] If $w^\eps$ converge weakly along $L_\eps$ to $w$ and $\|w^\eps\|_{X_\eps} \to \|w\|_{X_0}$, then $w^\eps$ converge strongly along $L_\eps$ to $w$. 
\item[(ii)]Suppose that $w^\eps$ converge strongly to $w$ along $L_\eps$ and $v^\eps$ converge weakly to $v$ along $L_\eps$. Then 
\begin{equation} \label{L_weak_strong} 
(w^\eps, v^\eps)_{X_\eps} \to (w,v)_{X_0}.
\end{equation} 
\item[(iii)]Let $a_1, a_2 \in \mathbb{R}$ and suppose that $w_1^\eps$, $w_2^\eps$ converge strongly (resp.\ weakly) to $w^1$, $w^2$ along $L_\eps$. Then $a_1 w_1^\eps + a_2 w_2^\eps$ converge strongly (resp.\ weakly) to $a_1 w_1 + a_2 w_2$.
\end{itemize} 
\end{proposition} 

In the proof of Proposition \ref{L_conv_prop} we use the following
\begin{lemma} \label{Lstar} 
Suppose that $L_\eps$ are asymptotically contractive. Whenever $(w^\eps)$, $w^\eps \in X^\eps$, converges strongly along $L_\eps$ (to any $w \in X_0$), we have
\[\|L_\eps^* L_\eps w^\eps - w^\eps \|_{X_\eps} \to 0, \]
where $L^*_\eps \colon X_0 \to X_\eps$ is the operator adjoint to $L^\eps$ for $\eps > 0$.  
\end{lemma} 
\begin{proof} 
We calculate 
\[\|L_\eps^* L_\eps w^\eps - w^\eps \|_{X_\eps}^2 = \|L_\eps^* L_\eps w^\eps\|_{X_\eps}^2 - 2 (L_*^\eps L_\eps w^\eps, w^\eps)_{X_\eps} + \|w^\eps\|_{X_\eps}^2.  \]
By strong convergence of $w^\eps$ along $L_\eps$, $\|w^\eps\|_{X_\eps} \to \|w\|_{X_0}$ and 
\[ (L_*^\eps L_\eps w^\eps, w^\eps)_{X_\eps} = (L_\eps w^\eps, L_\eps w^\eps)_{X_0} = \|L_\eps w^\eps\|_{X_0}^2 \to \|w\|_{X_0}^2 . \]
On the other hand, since $\|L^*_\eps\|_{B(X_0, X_\eps)} = \|L_\eps\|_{B(X_\eps, X_0)}$ for $\eps >0$, we have by \ref{H1}
\[ \limsup_{\eps\to 0^+} \|L_\eps^* L_\eps w^\eps\|_{X_\eps} \leq \limsup_{\eps\to 0^+}\|L_\eps\|_{B(X_\eps, X_0)}^2 \|w^\eps\|_{X_\eps} \leq \|w\|_{X_0}.  \]
Thus, 
\[ \limsup_{\eps\to 0^+} \|L_\eps^* L_\eps w^\eps - w^\eps \|_{X_\eps}^2 \leq \|w\|^2_{X_0} - 2 \|w\|^2_{X_0} + \|w\|^2_{X_0} = 0.\qedhere\]
\end{proof} 

\begin{proof}[Proof of Proposition \ref{L_conv_prop}]
(i) Using \ref{H1} we have 
\begin{equation*} 
\|w\|_{X_0} \leq \liminf_{\eps \to 0^+} \|L_\eps w^\eps\|_{X_\eps} \leq \limsup_{\eps \to 0^+} \|L_\eps w^\eps\|_{X_\eps} \leq \limsup_{\eps \to 0^+} \|w^\eps\|_{X_\eps} = \|w\|_{X_0},
\end{equation*} 
hence $\lim_{\eps \to 0^+} \|L_\eps w^\eps\|_{X_\eps} = \|w\|_{X_0}$. This implies that $L_\eps w^\eps \to w$ in $X_0$, which is the only thing we needed to check. 

(ii) Suppose that $w^\eps$ converge strongly to $w$ along $L_\eps$ and $v^\eps$ converge weakly to $v$ along $L_\eps$. We rewrite 
\[ (w^\eps, v^\eps)_{X_\eps} = (L_\eps^* L_\eps w^\eps, v^\eps)_{X_\eps} + (L_\eps^* L_\eps w^\eps - w^\eps, v^\eps)_{X_\eps}.  \]
We have 
\[ (L_\eps^* L_\eps w^\eps, v^\eps)_{X_\eps} = (L_\eps w^\eps, L_\eps v^\eps)_{X_0} \to (w, v)\]
and, using Lemma \ref{Lstar} together with asymptotic boundedness of $\|v^\eps\|_{X_\eps}$,  
\[|(L_\eps^* L_\eps w^\eps - w^\eps, v^\eps)_{X_\eps}|\leq \|L_\eps^* L_\eps w^\eps - w^\eps\|_{X_\eps} \| v^\eps\|_{X_\eps} \to 0  \]
which concludes the proof of \eqref{L_weak_strong}. 

(iii) Now take $a_1, a_2 \in \mathbb{R}$ and suppose that $w_1^\eps$, $w_2^\eps$ converge weakly to $w^1$, $w^2$ along $L_\eps$. Then clearly $\|a_1 w_1^\eps + a_2 w_2^\eps\|_{X_\eps}$ is asymptotically bounded and \[L_\eps(a_1 w_1^\eps + a_2 w_2^\eps) = a_1 L_\eps w_1^\eps + a_2 L_\eps w_2^\eps \rightharpoonup a_1 w_1 + a_2 w_2.
\]
If $w_1^\eps$, $w_2^\eps$ converge strongly to $w^1$, $w^2$ along $L_\eps$, then using part (ii) of the Proposition,
\begin{multline*} 
\|a_1 w_1^\eps + a_2 w_2^\eps\|_{X_\eps}^2 = |a_1|^2 \|w_1^\eps\|_{X_\eps}^2 + 2 a_1 a_2 (w_1^\eps, w_2^\eps)_{X_\eps} + |a_2|^2 \|w_2^\eps\|_{X_\eps}^2 \\ \to  |a_1|^2 \|w_1\|_{X_0}^2 + 2 a_1 a_2 (w_1, w_2)_{X_0} + |a_2|^2 \|w_2\|_{X_0}^2 = \|a_1 w_1 + a_2 w_2\|_{X_0}^2
\end{multline*} 
and 
\[L_\eps(a_1 w_1^\eps + a_2 w_2^\eps) = a_1 L_\eps w_1^\eps + a_2 L_\eps w_2^\eps \to a_1 w_1 + a_2 w_2.
\]
\end{proof} 

We observe that asymptotically contractive operators may be very far from either (asymptotically) injective or (asymptotically) surjective. For example, zero operators, i.e., $L_\eps \colon X_\eps \to X_0$ such that $L_\eps x_\eps = 0$ for $x_\eps \in X_\eps$, $\eps >0$ are always asymptotically contractive. Thus, the strength of particular results obtained from our theory (see Theorem \ref{thm:Mosco} below) depends on the family $L_\eps$. Allowing the lack of injectivity will often be natural in the examples. On the other hand, lack of asymptotic surjectivity might mean that the space $X_0$ is simply too large. Indeed, suppose that $X_\eps$ are connected to $X_0$ by $L_\eps$, and that $\iota \colon X_0\hookrightarrow \overline{X}_0$ is an embedding of $X_0$ into a larger Hilbert space $\overline X_0$. Then $X_\eps$ are connected to $\overline X_0$ by $\iota \circ L_\eps$. Moreover, if $L_\eps$ are asymptotically contractive, then so are $\iota \circ L_\eps$. The assumption of asymptotic surjectivity in the definition of Mosco convergence of spaces prevents this, leading to the following weak "uniqueness" property. 

\begin{proposition} \label{uniq} 
Suppose that $\iota \colon X_0 \hookrightarrow \overline X_0$ is an embedding of Hilbert spaces. If $\iota \circ L_\eps$ are asymptotically surjective, then $\iota$ is an isomorphism onto $\overline X_0$. 
\end{proposition} 

\begin{proof} 
By definition, $\iota$ is an isomorphism onto its image. We only need to show that the image is the whole $\overline X_0$. Let $\overline{w} \in \overline{X}_0$. By asymptotic surjectivity, there exists a sequence $(w^\eps)$, $w^\eps \in X_\eps$ such that 
\begin{equation} \label{iotaLeps} 
\iota \circ L_\eps w^\eps \rightharpoonup \overline w. 
\end{equation} 
In particular the sequence $(\iota \circ L_\eps w^\eps) \subset \overline{X}_0$ is bounded. As $\iota$ is an isomorphism onto its image, also $(L_\eps w^\eps) \subset X_0$ is bounded. Therefore, there exists a subsequence $(\eps_k)$ and $w \in X_0$ such that $L_\eps w^\eps_k \rightharpoonup w$ in $X_0$. By continuity of $\iota$ and \eqref{iotaLeps}, $\overline{w} = \iota\,w$. 
\end{proof}
We note that in the condition \ref{H2} defining asymptotic surjectivity we could also require strong convergence $L_\eps w^\eps \to w$, or even weak/strong convergence of $w^\eps$ along $L_\eps$. However, weak convergence of $L_\eps u^\eps$ is enough for Proposition \ref{uniq} to hold.

The present concept of Mosco convergence of Hilbert spaces is not comparable to (pointed) Gromov--Hausdorff convergence. The latter is a convergence in the space of \emph{isometry classes} of (locally) compact metric spaces. It is usually defined as convergence in the metric arising as the infimum of Hausdorff distances between images under isometric embeddings into a common space. However, we can give a sufficient condition for Mosco convergence in a similar language. 

\begin{proposition} 
	Suppose that there exists a Hilbert space $\overline{X}$ and isometries $\iota_0 \colon X_0 \to \overline{X}$, $\iota_\eps \colon X_\eps \to \overline{X}$, $\eps >0$. Let $P$ denote the orthogonal projection of $\overline{X}$ onto $\iota_0\,X_0$. Then $X_\eps$ are connected to $X_0$ by $L_\eps\! :=\iota_0^{-1}\circ P\circ \iota_\eps$. 

    If moreover for every $x_0 \in X_0$ there exists a sequence $x_\eps \in X_\eps$ such that $\iota_\eps x_\eps \rightharpoonup \iota_0 x_0$ in  $\overline{X}$, then $X_\eps$ Mosco-converge to $X_0$ along $L_\eps$. 	
\end{proposition} 

\begin{proof} 
	Condition \ref{H1} is clearly satisfied, as $\iota_0^{-1}$, $L_\eps$, $\iota_\eps$ all have norms bounded by $1$.  As for the second part, we have
	\[L_{\eps} x_{\eps} = \iota_0^{-1}\, P\, \iota_\eps x_\eps \rightharpoonup \iota_0^{-1}\, P\, \iota_0 x_0 = x_0.\]
\end{proof}

We recall that any convex, lower semicontinuous functional $\cE\colon X \to [0,\infty]$ on a Hilbert space $X$ generates a gradient flow \cite[Theorem 3.2]{Brezis}. That is, for any $u_0 \in \overline{D(\cE)}$, $T>0$, there exists exactly one $u\in C([0,T],X) \cap W^{1,1}_{loc}(]0,T])$ satisfying 
\[ u_t \in - \partial \cE(u) \quad \text{for a.\ e. } t \in ]0,T[, \qquad u(0) = u_0.\]
We call the resulting map $S\colon \overline{D(\cE)} \to C([0,T],X)$ the gradient flow of $\cE$. 
\begin{definition}\label{con-flo}
	Suppose that $X_\eps$ converge to $X_0$ along $L_\eps$. Let $S_0$, $S_\eps$ be gradient flows of $\cE_0 \colon X_0 \to [0, \infty]$, $\cE_\eps \colon X_\eps \to [0, \infty]$ respectively. We say that $S_\eps$ uniformly converge to $S_0$ along $L_\eps$ if 
	\[L_\eps S_{\eps} u_0^\eps \to S_0 u_0 \quad \text{in } C([0,T],X_0) \quad \text{and} \quad \|S_{\eps} u_0^\eps\|_{X_\eps} \to \|S_0 u_0\|_{X_0} \text{in } C([0,T])\]
	whenever $u_0^\eps$ converge strongly to $u_0$ along $L_\eps$, $u_0 \in \overline{D(\cE_0)}$, $u_0^\eps \in \overline{D(\cE_\eps)}$ for $\eps >0$.  
\end{definition} 

Here comes the generalization of the Mosco-convergence of functionals to the case where the underlying function spaces 
are connected by $L_\eps$. We notice that if $X_\eps = X$ and $L_\eps = Id$, then the new notion coincides with the classical one, see \cite{attouch}.
\begin{definition} \label{mos-con}
    Let $\cE_0 \colon X_0 \to [0, \infty]$, $\cE_\eps \colon X_\eps \to [0, \infty]$ be convex and lower semicontinuous. We say that $\cE_\eps$ Mosco-converge to $\cE_0$ along $L_\eps$ if 
	\begin{enumerate}[label={\rm (H\arabic*)}]
    \setcounter{enumi}{2}
		\item \label{H3} whenever $\limsup_{\eps \to 0^+} \|w^\eps\|_{X_{\eps}} < \infty$ and $L_\eps w^\eps \rightharpoonup w$ in $X_0$,  
		\[\liminf_{\eps \to 0^+} \cE_\eps(w^\eps) \geq \cE_0(w),\]
		\item \label{H4} for every $w \in D(\cE_0)$ there exists $(w^\eps)$, $w^\eps \in X_\eps$ such that 
  \[ L_\eps w^\eps \to w \text{ in } X_0, \quad \|w^\eps\|_{X_\eps}\to \|w\|_{X_0},\quad \cE_\eps(w^\eps) \to \cE_0(w).\]
	\end{enumerate} 
\end{definition}
\noindent 

We are now ready to state our main abstract theorem. It extends the classical result of Mosco \cite{mosco}, see also \cite{watanabe}. 

\begin{theorem} \label{thm:Mosco} Suppose that $X_\eps$ converge to $X_0$ along $L_\eps$. Let $S_\eps$, $S_0$ be the gradient flows of $\cE_\eps\colon X_\eps \to [0, \infty]$, $\cE_0 \colon X_0 \to [0, \infty]$ respectively. If $\cE_\eps$ Mosco-converge to $\cE_0$ along $L_\eps$, then $S_\eps$ uniformly converge to $S_0$ along $L_\eps$.
\end{theorem} \noindent
In other words, given a family of bounded linear operators $L_\eps \colon X_\eps \to X_0$ and functionals $\cE_\eps\colon X_\eps \to [0, \infty]$, $\cE_0 \colon X_0 \to [0, \infty]$, in order to show convergence of their gradient flows along $L_\eps$, one needs to check assumptions \ref{H1}--\ref{H4}. 

The proof of Theorem \ref{thm:Mosco} is presented in a separate section.
\section{Proof of the abstract convergence theorem} \label{s3}
The proof is based on the \emph{minimizing movements scheme} which we will now recall. Let $X$ be a Hilbert space and $\cE \colon X \to [0, \infty]$ a convex, proper, lower semicontinuous functional on $X$. Given $w \in X$, $\lambda >0$ we define for $v \in X$ 
\[ \cE_w^\lambda(v) = \lambda \cE(v) + \frac{1}{2}\|w-v\|^2.\]
The resulting strictly convex, proper, lower semicontinuous functional $\cE_w^\lambda$ on $X$ has a unique minimizer, which we denote $\cJ_\lambda(w)$ ($\cJ_\lambda$ is the \emph{resolvent operator} associated with $\cE$). Given $u_0 \in X$, we define 
\[u^N_k := \left(\cJ_\frac{T}{N}\right)^k(u_0)\] 
for $N \in \bN$, $k=0,\ldots, N-1$. We then construct a minimizing movements approximation $u^N \in C([0,T],X)$ of $u$ as the affine interpolation of $u^N_k$, i.\,e. 
\[ u^N(t) = (k+1 - Nt/T) u^N_k + (Nt/T-k) u^N_{k+1} \quad \text{for } t \in [Tk/N, T(k+1)/N[, \quad k=0,\ldots, N-1. \]
Then $u^N$ converges uniformly to the trajectory $u$. If $u_0$ in $D(\cE)$, then we have a quantitative estimate \cite{NSV} 
\[\max_{t \in [0,T]} \|u(t) - u^N(t)\|^2_X \leq \cE(u_0)/N. \]

In our setting, we deal with a sequence of functionals $\cE_\eps \colon X_\eps \to [0, \infty[$, $\cE_0 \colon X_0 \to [0, \infty[$. We denote the associated resolvent operators by $\cJ_\lambda^\eps$, $\cJ_\lambda^0$. We have:
\begin{lemma}\label{lem:res} 
Suppoose that $X_\eps$ are Hilbert spaces connected to $X_0$ by asymptotically contractive operators $L_\eps$, and $\cE_\eps$ Mosco-converge to $\cE$ along $L_\eps$. Let $\lambda >0$. If $w^\eps$ converges strongly to $w$ along $L_\eps$, then $\cJ_\lambda^\eps(w_\eps)$ strongly converges to $\cJ_\lambda^0(w)$ along $L_\eps$. Moreover, $\cE_\eps(\cJ_\lambda^\eps(w_\eps))\to \cE_0(\cJ_\lambda^0(w))$ as $\eps \to 0^+$. 
\end{lemma}

\begin{proof} 
Suppose that $w^\eps$ converges strongly to $w$ along $L_\eps$. Let $v_* \in D(\cE_0)$. By Mosco-convergence, there exists a sequence $v_*^\eps$ strongly convergent to $v_*$ along $L_\eps$ and such that $\cE_\eps(v_*^\eps) \to \cE(v_*)$. By definition of $\cJ_\lambda^\eps$,  
\begin{equation} \label{Moreau_comp}
\lambda \cE_\varepsilon(\cJ_\lambda^\eps(w^\eps)) + \tfrac{1}{2}\|w^\eps- \cJ_\lambda^\eps(w^\eps)\|_{X_\eps}^2 \leq \lambda \cE_\eps(v_*^\eps) + \tfrac{1}{2}\|w^\eps-v_*^\eps\|_{X_\eps}^2.
\end{equation} 
Since the r.\,h.\,s.\ is asymptotically bounded, we deduce that $\|w^\eps- \cJ_\lambda^\eps(w^\eps)\|_{X_\eps}$ and therefore also $\|\cJ_\lambda^\eps(w^\eps)\|_{X_\eps}$ is asymptotically bounded. Thus, there is a subsequence $(\eps_k)$, $\eps_k \to 0$ and $v \in X_0$ such that $L_{\eps_k} \cJ_\lambda^{\eps_k}(w^{\eps_k}) \rightharpoonup v$ in $X_0$ as $k \to \infty$. By Proposition \ref{L_conv_prop}(iii), $w^\eps-v_*^\eps$ converges strongly to $w - v_*$ along $L_\eps$. Using \ref{H1} and \ref{H3}, we now pass to the limit with \eqref{Moreau_comp} along the sequence $\eps_k$:
\begin{multline}\label{Mosco_est_1} \lambda \cE_0(v) + \tfrac{1}{2}\|w- v\|_{X_0}^2 
 \leq \liminf_{k \to \infty} \lambda \cE_{\varepsilon_k}\left(\cJ_\lambda^{\eps_k}(w^{\eps_k})\right) + \liminf_{k \to \infty}\tfrac{1}{2}\|L_{\eps_k}w^{\eps_k}- L_{\eps_k}\cJ_\lambda^{\eps_k}(w^{\eps_k})\|_{X_0}^2\\ \leq \liminf_{k \to \infty} \lambda \cE_{\varepsilon_k}\left(\cJ_\lambda^{\eps_k} (w^{\eps_k})\right) + \liminf_{k \to \infty}  \tfrac{1}{2}\|w^{\eps_k}- \cJ_\lambda^{\eps_k}(w^{\eps_k})\|_{X_{\eps_k}}^2 \\ \leq \lim_{\eps\to 0^+}\left(\lambda \cE_\eps(v_*^\eps) + \tfrac{1}{2}\|w^\eps-v_*^\eps\|_{X_\eps}^2\right) = \lambda \cE_0(v_*) + \tfrac{1}{2}\|w-v_*\|_{X_0}^2 . 
\end{multline} 
Since $v_*$ is an arbitrary element of $D(\cE)$, we infer that $v = \cJ_\lambda^0(w)$. By the usual argument involving uniqueness of the limit, we actually have $L_\eps \cJ_\lambda^\eps(w^\eps) \rightharpoonup \cJ_\lambda^0(w)$, without restriction to a subsequence. 

Next, choosing $v_* = \cJ_\lambda^0(w)$ and repeating the sequence of estimates \eqref{Mosco_est_1} in a slightly more careful manner, we get from \eqref{Moreau_comp}: 
\begin{multline} \label{Mosco_est_2} \lambda \cE_0(\cJ_\lambda^0(w)) + \tfrac{1}{2}\|w- \cJ_\lambda^0(w)\|_{X_0}^2 
 \leq \liminf_{\eps \to 0^+} \lambda \cE_{\varepsilon}\left(\cJ_\lambda^{\eps}(w^{\eps})\right) + \liminf_{\eps \to 0^+}\tfrac{1}{2}\|L_{\eps}w^{\eps}- L_{\eps}\cJ_\lambda^{\eps}(w^{\eps})\|_{X_0}^2\\ \leq \liminf_{\eps \to 0^+} \lambda \cE_{\varepsilon}\left(\cJ_\lambda^{\eps}(w^{\eps})\right) + \liminf_{\eps \to 0^+}\tfrac{1}{2}\|w^{\eps}- \cJ_\lambda^{\eps}(w^{\eps})\|_{X_\eps}^2\\ \leq \limsup_{\eps \to 0^+} \left(\lambda \cE_{\varepsilon}\left(\cJ_\lambda^{\eps}(w^{\eps})\right) + \tfrac{1}{2}\|w^{\eps}- \cJ_\lambda^{\eps}(w^{\eps})\|_{X_\eps}^2\right)  \\ \leq \lim_{\eps\to 0^+}\left(\lambda \cE_\eps(v_*^\eps) + \tfrac{1}{2}\|w^\eps-v_*^\eps\|_{X_\eps}^2\right) = \lambda \cE_0(\cJ_\lambda^0(w)) + \tfrac{1}{2}\|w-\cJ_\lambda^0(w)\|_{X_0}^2 . 
\end{multline} 
We note that for any two sequences of real numbers $a_\eps$, $b_\eps$, 
\[ a+b \leq \liminf_{\eps \to 0^+} a_\eps + \liminf_{\eps \to 0^+} b_\eps, \quad   \limsup_{\eps \to 0^+} \left(a_\eps + b_\eps\right) \leq a+b\ \implies\ \lim_{\eps \to 0^+} a_\eps = a, \quad \lim_{\eps \to 0^+} b_\eps = b.\]
Thus, we deduce from \eqref{Mosco_est_2} that 
\[\lim_{\eps \to 0^+} \cE_{\varepsilon}\left(\cJ_\lambda^{\eps}(w^{\eps})\right) = \cE_0(\cJ_\lambda^0(w)) \quad \text{and} \quad \lim_{\eps \to 0^+} \|w^{\eps}- \cJ_\lambda^{\eps}(w^{\eps})\|_{X_\eps} = \|w-\cJ_\lambda^0(w)\|_{X_0}.\]
By Proposition \ref{L_conv_prop}(i), $w^\eps- \cJ_\lambda^\eps(w^\eps)$ converges strongly to $w- \cJ_\lambda^0(w)$ along $L_\eps$. Finally, by Proposition \ref{L_conv_prop}(iii), $\cJ_\lambda^\eps(w^\eps)$ converges strongly to $\cJ_\lambda^0(w)$ along $L_\eps$.
\end{proof} 

\begin{proof}[Proof of Theorem \ref{thm:Mosco}] 
Given $\lambda >0$, let us denote by $u^\lambda$ the trajectory of $S_0$ emanating from $\cJ_\lambda u_0$, and by $u^{\lambda, N}$, the minimizing movements approximation of $u^\lambda$. Similarly, for $\eps>0$, we denote by $u^{\eps, \lambda}$ the trajectory of $S_\eps$ emanating from $\cJ_\lambda^\eps u_0^\eps$, and by $u^{\eps, \lambda, N}$, its minimizing movements approximation. We then have for a fixed $t \in ]0,T]$ 
\begin{multline*}
\|L_\eps u^\eps - u\|_{X_0} \\ \leq \|L_\eps u^\eps - L_\eps u^{\eps,\lambda}\|_{X_0} + \|L_\eps u^{\eps, \lambda} - L_\eps u^{\eps,\lambda,N}\|_{X_0} + \|L_\eps u^{\eps, \lambda, N} - u^{\lambda,N}\|_{X_0} + \|u^{\lambda,N} - u^\lambda \|_{X_0} + \|u^\lambda - u \|_{X_0}  \\ \leq \|L_\eps\|_{B(X_\eps,X_0)} \left(
\|u^\eps - u^{\eps,\lambda}\|_{X_\eps} + \|u^{\eps, \lambda} - u^{\eps,\lambda,N}\|_{X_\eps}\right) \\ + \|L_\eps u^{\eps, \lambda, N} - u^{\lambda,N}\|_{X_0} + \|u^{\lambda,N} - u^\lambda \|_{X_0} + \|u^\lambda - u \|_{X_0} \\
\leq \|L_\eps\|_{B(X_\eps,X_0)} \left(
\|u^\eps_0 - \cJ_\lambda^\eps(u^\eps_0)\|_{X_\eps} + \sqrt{\cE_\eps(\cJ_\lambda^\eps(u^\eps_0))/N}\right) \\+ \|L_\eps u^{\eps, \lambda, N} - u^{\lambda,N}\|_{X_0} + \sqrt{\cE_0(\cJ_\lambda^0(u_0))/N} + \|\cJ_\lambda^0(u_0) - u_0 \|_{X_0}
\end{multline*} 
Let $\delta >0$. First we fix $\lambda>0$ small enough such that $\|\cJ_\lambda^0(u_0) - u_0 \|_{X_0} < \delta$. Then we choose $N \in \bN$ so that $\sqrt{\cE_0(\cJ_\lambda^0(u_0))/N} < \delta$. By Lemma \ref{lem:res}, we then have for sufficiently small $\eps>0$  
\[ \|u^\eps_0 - \cJ_\lambda^\eps(u^\eps_0)\|_{X_\eps} + \sqrt{\cE_\eps(\cJ_\lambda^\eps(u^\eps_0))/N} < 3\delta.\]
Iterated application of Lemma \ref{lem:res} yields also strong convergence along $L_\eps$ of $(\cJ_\frac{T}{N}^\eps)^k(\cJ_\lambda^\eps(u^\eps_0))$ to $(\cJ_\frac{T}{N}^0)^k(\cJ_\lambda^0(u_0))$ for $k = 0, \ldots,N-1$. Therefore in particular 
\[\|L_\eps u^{\eps, \lambda, N} - u^{\lambda,N}\|_{X_0} \leq \delta\]
for sufficiently small $\eps >0$, independently of $t$. We also recall that by \ref{H1}, $\|L_\eps\|_{B(X_\eps,X_0)} \leq 1+ \delta$ for sufficiently small $\eps >0$. Thus, we get $ L_\eps u^\eps \to u$ in $C([0,T],X_0)$. 

Next, we write 
\begin{multline*}
\left|\| u^\eps\|_{X_\eps} - \|u\|_{X_0}\right| \\ \leq \left|\|u^\eps\|_{X_\eps} - \|u^{\eps,\lambda}\|_{X_\eps}\right| + \left|\|u^{\eps, \lambda}\|_{X_\eps} - \|u^{\eps,\lambda,N}\|_{X_\eps}\right| + \left|\|u^{\eps, \lambda, N}\|_{X_\eps} - \|u^{\lambda,N}\|_{X_0}\right| \\+ \left|\|u^{\lambda,N}\|_{X_0} - \|u^\lambda \|_{X_0}\right| + \left|\|u^\lambda\|_{X_0} - \|u \|_{X_0}\right| \\ \leq 
\|u^\eps - u^{\eps,\lambda}\|_{X_\eps} + \|u^{\eps, \lambda} - u^{\eps,\lambda,N}\|_{X_\eps} + \left|\|u^{\eps, \lambda, N}\|_{X_\eps} - \|u^{\lambda,N}\|_{X_0}\right| + \|u^{\lambda,N} - u^\lambda \|_{X_0} + \|u^\lambda - u \|_{X_0} \\
\leq 
\|u^\eps_0 - \cJ_\lambda^\eps(u^\eps_0)\|_{X_\eps} + \sqrt{\cE_\eps(\cJ_\lambda^\eps(u^\eps_0))/N} + \left|\|u^{\eps, \lambda, N}\|_{X_\eps} - \|u^{\lambda,N}\|_{X_0}\right| \\+ \sqrt{\cE_0(\cJ_\lambda^0(u_0))/N} + \|\cJ_\lambda^0(u_0) - u_0 \|_{X_0}.
\end{multline*} 
We have already estimated all the terms on the r.\,h.\,s., except for the third one. Its asymptotic smallness follows again from the strong convergence of $(\cJ_\frac{T}{N}^\eps)^k(\cJ_\lambda^\eps(u^\eps_0))$ to $(\cJ_\frac{T}{N}^0)^k(\cJ_\lambda^0(u_0))$ along $L_\eps$ for $k = 0, \ldots,N-1$. Thus we obtain convergence $\| u^\eps\|_{X_\eps} \to \|u\|_{X_0}$ in $C([0,T])$ and conclude the proof.
\end{proof} 

As a step of the proof, we have proved:
\begin{corollary} \label{CRe}
    Suppose that $X_\eps$ Mosco-converge to $X_0$ along $L_\eps$, and $\cE_\eps$ Mosco-converge to $\cE$ along $L_\eps$. Let $S_\varepsilon$, $S_0$ be the gradient flow of $\mathcal{E}_\varepsilon$, $\mathcal{E}_0$ respectively.
    Assume that for a fixed $\lambda>0$, $\mathcal{J}_\lambda^\varepsilon(w_\varepsilon)$ strongly converges to  $\mathcal{J}_\lambda^0(w)$ along $L_\varepsilon$, and $\mathcal{E}_\varepsilon\left(\mathcal{J}_\lambda^\varepsilon(w_\varepsilon)\right)$ is bounded whenever $w_\varepsilon$ strongly converges to $w$ along $L_\varepsilon$ (as $\varepsilon\downarrow0$).
    Then $S_\varepsilon$ uniformly converges to $S_0$ along $L_\varepsilon$.
\end{corollary}
In the case $X_\varepsilon=X_0$, $L_\varepsilon=\mathrm{id}$, the convergence of the resolvent $\mathcal{J}_\lambda^\varepsilon$ implies the convergence of trajectories.
This type of results is well established even when $X$ is a Banach space and the trajectory is not necessarily a gradient flow but a flow generated by a maximal monotone operator; see e.g.\ \cite{BP}. 
Instead of the minimizing movement approximation $u^N$, in \cite{BP} a simpler approximation
\[
    \bar{u}^N(t) = (\mathcal{J}_{t/N})^N (u_0)
\]
is used.
We argue that 
to estimate $\|L_\varepsilon u^\varepsilon-u\|_{X_0}$ we may use 
$\bar{u}^{\lambda,N}$, $\bar{u}^{\varepsilon,\lambda,N}$ instead of $u^{\lambda,N}$, $u^{\varepsilon,\lambda,N}$.
The only difference is the bound of 
\[
    \| u^{\varepsilon,\lambda} - \bar{u}^{\varepsilon,\lambda,N} \|_{X_\varepsilon}
    \quad\text{and}\quad
    \| u^\lambda - \bar{u}^{\lambda,N}\|_{X_0}.
\]
In \cite{BP}, the estimate 
\[
    \left\| u - \bar{u}^N(t) \right\|_{X_0} \le 2tN^{-1/2}
    \left\| \partial^0\mathcal{E}_0(u_0) \right\|_{X_0}
\]
due to \cite[Theorem I]{CL} is invoked instead of 
\[
    \| u - \bar{u}^N(t) \|_{X_0}
    \le \left(\mathcal{E}_0(u_0)/N \right)^{1/2}.
\]
Thus, if $\|\partial^0\mathcal{E}_\varepsilon\left(\mathcal{J}_\lambda^\varepsilon(w_\varepsilon)\right)\|_{X_\eps}$ is bounded by $M_\lambda$ as $\varepsilon \to 0^+$ for a fixed $\lambda>0$ when $w_\varepsilon$ strongly converges to $w$ along $L_\varepsilon$, then we have
\[
    \| u^{\varepsilon,\lambda} - \bar{u}^{\varepsilon,\lambda,N} \|_{X_\varepsilon}
    \le 2tN^{-1/2} M_\lambda
\]
Here $\partial^0\mathcal{E}_\varepsilon$ denotes the minimal section of the subdifferential $\partial\mathcal{E}_\varepsilon$.
Since $\mathcal{J}_\lambda^0(u_0)\in D(\partial\mathcal{E}_0)$, we see that
\[
    \| u^\lambda - \bar{u}^{\lambda,N} \|_{X_0}
    \le 2tN^{-1/2}
    \left\| \partial^0\mathcal{E}_0(\mathcal{J}_\lambda^0(u_0)) \right\|_{X_0} 
    \le 2tN^{-1/2} M'_\lambda
\]
with some finite $M'_\lambda$.
We argue as before to get
\begin{multline*}
    \| L_\varepsilon u^\varepsilon -u \|_{X_0}
    \le \| L_\varepsilon \|_{B(X_\varepsilon,X_0)}
    \left( \left\| u_0^\varepsilon - \mathcal{J}_\lambda^\varepsilon(u_0^\varepsilon) \right\|_{X_\varepsilon}
    + \frac{M_\lambda T}{N^{1/2}} \right) \\
    + \| L_\varepsilon \bar{u}^{\varepsilon,\lambda,N} -\bar{u}^{\lambda,N}\|_{X_0}
    + \frac{M'_\lambda T}{N^{1/2}}
    + \left\| u_0 - \mathcal{J}_\lambda^0(u_0) \right\|_{X_0}.
\end{multline*}
The rest of the argument is the same.
We thus observe that the boundedness assumption of $\mathcal{E}_\varepsilon\left(\mathcal{J}_\lambda^\varepsilon(w_\varepsilon)\right)$ is replaced by a boundedness assumption of $\|\partial^0\mathcal{E}_\varepsilon\left(\mathcal{J}_\lambda^\varepsilon(w_\varepsilon)\right)\|_{X_\eps}$ as $\varepsilon \to 0^+$ for a fixed $\lambda>0$.

\bigskip

In the following sections we are going to present several applications of our abstract convergence theorem. In order to show flexibility of our tool, we will deal with the total variation flow, the flow of the $p$-heat eq., a flow on graphs and on thin domains.

\section{Equations on a thin domain} \label{Sthin}

As a first example, we want to derive the limiting form of a parabolic equation with respect to a changing spatial domain. We have in mind  a slab of variable height, whose thickness tends to $0$. More precisely, let 
$\omega$ be a bounded Lipschitz domain in $\bR^{n-1}$, and let $\Omega_\varepsilon$ be defined by 
\[
    \Omega_\varepsilon = \left\{ (x', x_n) \bigm|
    x' \in \omega,\ 
    \varepsilon g_-(x') < x_n < \varepsilon g_+(x') \right\},
\]
where $g_\pm\in\operatorname{Lip}(\omega)$ are such that $g_-(x')<g_+(x')$ for all $x'\in\omega$.
Given $p>1$ (see the end of this section for the case $p=1$), we consider the $p$-heat equation with the Neumann boundary condition:
 \begin{align}
 \begin{aligned} \label{EpH}
     &u_t = \operatorname{div}\left( |\nabla u|^{p-2} \nabla u \right) &\quad \text{in} &\quad ]0,T[\times \Omega_\varepsilon, \\
     &\nabla u \cdot \nu^{\Omega_\varepsilon} = 0 &\text{on} &\quad ]0,T[ \times \partial\Omega_\varepsilon.
 \end{aligned}
 \end{align}
We take a Hilbert space $X_\varepsilon=L^2(\Omega_\varepsilon)$ equipped with the inner product
\[
    (u,v)_{X_\eps} := \frac{1}{\varepsilon} \int_{\Omega_\varepsilon} u\, v.
\]
Equation \eqref{EpH} can be written as the gradient flow $S_\varepsilon$ of the functional $\cE_\eps \colon X_\eps \to [0, \infty]$ given by
\begin{equation} \label{E_thin_p>1} 
    \mathcal{E}_\varepsilon(u) = \left \{
	\begin{array}{cl}
        \displaystyle \frac{1}{\varepsilon p} \int_{\Omega_\varepsilon} |\nabla u|^p &\text{if } u\in L^2(\Omega_\varepsilon)\cap W^{1,p}(\Omega_\varepsilon) \\
        \infty, &\text{if } u\in L^2(\Omega_\varepsilon) \backslash W^{1,p} (\Omega_\varepsilon).
	\end{array}
	\right.
\end{equation}

We set $X_0=L^2(\omega)$, which is equipped with a weighted inner product
\[
    (v,w)_{X_0} = \int_{\omega} g\, v\, w 
    \quad v,w \in X_0,
\]
where $g(x')=g_+(x')-g_-(x')$. We shall show that $S_\eps$ converges to the gradient flow $S_0$ of the functional $\cE_0 \colon X_0 \to [0,\infty]$ given by
\begin{equation} \label{E0_thin} 
    \mathcal{E}_0(v) = \left \{
	\begin{array}{cl}
        \displaystyle \frac{1}{p} \int_{\omega} g\, |\nabla' v|^p &\text{if } v\in L^2(\omega)\cap W^{1,p}(\omega), \\
        \infty &\text{if } v\in L^2(\omega) \backslash W^{1,p}(\omega). 
	\end{array}
	\right.
\end{equation}
Here $\nabla'v=(\partial v/\partial x_1,\ldots,\partial v/\partial x_{n-1})$ . We observe that $S_0$ corresponds to the equation
 \begin{align}
 \begin{aligned} 
     &gv_t = \operatorname{div}' \left( g|\nabla' v|^{p-2} \nabla' v \right) &\quad \text{in} &\quad ]0,T[ \times \omega_\varepsilon, \\
     &\quad \nu^\omega \cdot \nabla' v = 0 &\text{on} &\quad ]0,T[ \times \partial\omega_\varepsilon,
 \end{aligned}
 \end{align}
where $\operatorname{div}'$ denotes the divergence on $\omega$.

We define connecting operators $L_\varepsilon$ by
\[
    L_\varepsilon u(x')
    = \frac{1}{\varepsilon g(x')} \int_{\varepsilon g_-(x')}^{\varepsilon g_+(x')} u(x',x_n)\, \dd x_n,
\]
i.e., $L_\varepsilon$ is a vertically averaging operator.
\begin{lemma} \label{lem:Xthin}
$X_\eps$ Mosco-converge to $X_0$ along $L_\eps$.  
\end{lemma} 

\begin{proof} 
We first check that $L_\eps$ are asymptotically contractive. By changing the variable of integration, we see that
\begin{align} \label{thin_change_var}
    (L_\varepsilon f)(x')
    &= \frac{1}{\varepsilon g(x')} \int_{\varepsilon g_-(x')}^{\varepsilon g_+(x')} f(x',x_n)\, \dd x_n\nonumber \\
    &= \frac{1}{\varepsilon g(x')} \int_0^{\varepsilon g(x')} f\left(x',y_n+\varepsilon g_-(x')\right)\, \dd y_n \\
    &= \int_0^1 f\left(x',z_n \varepsilon g(x') + \varepsilon g_-(x') \right)\, \dd z_n.\nonumber
\end{align}
By Jensen's inequality, 
\[
    \|L_\varepsilon f\|_{X_0}^2
    \leq \int_0^1 \left\{ \int_{\omega} \left| f\left(x', z_n\varepsilon g(x') + \varepsilon g_-(x') \right)\right|^2 g(x')\, \dd x' \right\} \dd z_n.
\]
By changing the variable $x_n=z_n\varepsilon g(x')+\varepsilon g_-(x')$, the right-hand side equals
\[
    \frac{1}{\varepsilon} \int_{\omega} \int_{\varepsilon g_-(x')}^{\varepsilon g_+(x')} \left|f(x',x_n)\right|^2 \dd x'\dd x_n.
\]
We thus conclude that
\[
    \| L_\varepsilon f \|_{X_0}^2 \leq \|f\|_{X_\varepsilon}^2,
\]
which yields \ref{H1}.

In order to show \ref{H2}, for a given $w \in X_0$ we take $w^\eps \in X_\eps$ given by 
\[ w^\eps(x', x_n) = w(x') \quad \text{ for } (x', x_n) \in \Omega_\eps.\]
Then $L_\eps w^\eps = w$ is a constant sequence, which clearly converges to $w$ in $X_0$.  
\end{proof} 

\begin{lemma} \label{lem:Ethin}
$\cE_\eps$ Mosco-converge to $\cE_0$ along $L_\eps$.  
\end{lemma} 

\begin{proof}
Let $(w^\eps)$ be such that $L_\varepsilon w^\varepsilon \rightharpoonup w$ in $X_0$ and $\sup_{0<\varepsilon<1}\|w^\varepsilon\|_{X_\varepsilon}<\infty$.  Suppose that $w^\eps\in W^{1,p}(\Omega_\eps)$ for a fixed $\eps > 0$. Since the mapping $G_\eps\colon \omega\times]0,1[ \to \Omega_\varepsilon$ given by  $G_\eps(x', z_n) = (x', z_n \eps g(x') + \eps g_-(x'))$ is bi-Lipschitz, by \cite[Theorem 2.2.2] {ziemer} the composition $w^\eps \circ G_\eps$ belongs to $W^{1,p}(\omega\times]0,1[)$ and its derivative is given by the usual chain rule.
Recalling \eqref{thin_change_var}, we calculate 
\begin{multline} \label{thin_grad_calc}
    \nabla' L_\eps w^\eps (x')\\ = \int_0^1 (\nabla' w^\eps) \left(x',z_n \varepsilon g(x') + \varepsilon g_-(x') \right) + \frac{\partial w^\eps}{\partial x_n} \left(x',z_n \varepsilon g(x') + \varepsilon g_-(x') \right) \left(\eps z_n \nabla' g(x') + \eps \nabla' g_-(x')\right) \dd z_n \\ 
    = \frac{1}{\varepsilon g(x')} \int_{\varepsilon g_-(x')}^{\varepsilon g_+(x')} (\nabla' w^\eps) \left(x',x_n\right) + \frac{\partial w^\eps}{\partial x_n} \left(x',x_n\right) \left(\frac{x_n - \eps g_-(x')}{g(x')} \nabla' g(x') + \eps \nabla ' g_-(x')\right) \dd x_n.
\end{multline} 
We note that 
\begin{multline*} 
\left|\frac{x_n - \eps g_-(x')}{g(x')} \nabla' g(x') + \eps \nabla ' g_-(x')\right| = \left|\frac{x_n - \eps g_-(x')}{g(x')} \nabla' g_+(x') + \frac{\eps g_+(x') - x_n}{g(x')} \nabla ' g_-(x')\right| \\ \leq \eps |\nabla ' g_+(x')| + \eps |\nabla ' g_-(x')|
\end{multline*}
Thus, by Jensen's and Young's inequalities, for any $\vartheta >0$, 
\begin{multline*} 
 \left|\nabla' L_\eps w^\eps (x')\right|^p \\ \leq 
     \frac{1}{\varepsilon g(x')} \int_{\varepsilon g_-(x')}^{\varepsilon g_+(x')} \left|\nabla' w^\eps (x',x_n) + \frac{\partial w^\eps}{\partial x_n} (x',x_n) \left(\frac{x_n - \eps g_-(x')}{g(x')} \nabla' g(x') + \eps \nabla ' g_-(x')\right)\right|^p  \dd x_n  \\
     \leq 
     \frac{1}{\varepsilon g(x')}\! \int_{ \raisebox{-3pt}{\scriptsize$\varepsilon g_-(x')$}}^{ \raisebox{3pt}{\scriptsize$\varepsilon g_+(x')$}} \hspace{-22pt}(1 + \vartheta)\left|\nabla' w^\eps (x',x_n)\right|^p + C(\vartheta)\left|\frac{\partial w^\eps}{\partial x_n} (x',x_n) \left(\frac{x_n - \eps g_-(x')}{g(x')} \nabla' g(x') + \eps \nabla ' g_-(x')\right)\right|^p  \!\dd x_n \\
     \leq \frac{1}{\varepsilon g(x')} \int_{\varepsilon g_-(x')}^{\varepsilon g_+(x')} (1 + \vartheta)\left|\nabla' w^\eps (x',x_n)\right|^p + C(\vartheta)\eps^p \left( |\nabla ' g_+(x')| +  |\nabla ' g_-(x')|\right)^p \left|\frac{\partial w^\eps}{\partial x_n} (x',x_n)\right|^p \dd x_n \\ 
     \leq \frac{1}{\varepsilon g(x')} \int_{\varepsilon g_-(x')}^{\varepsilon g_+(x')}\left(1 + \vartheta + C(\vartheta)\eps^p \left( |\nabla ' g_+(x')| +  |\nabla ' g_-(x')|\right)^p\right) \left|\nabla w^\eps (x',x_n)\right|^p \dd x_n. 
\end{multline*} 
Multiplying the obtained inequality by $g(x')$ and integrating it over $\omega$, we get 
\begin{multline*}p\mathcal{E}_0(L_\varepsilon w^\varepsilon) = \int_{\omega} g(x') \left|\nabla' L_\eps w^\eps (x')\right|^p \dd x'\\ \leq \left(1 + \vartheta + C(\vartheta)\eps^p \left( \|\nabla ' g_+\|_{L^\infty(\omega)^n} +  \|\nabla ' g_-\|_{L^\infty(\omega)^n}\right)^p\right) \frac{1}{\varepsilon} \int_{\omega}  \int_{\varepsilon g_-(x')}^{\varepsilon g_+(x')} \left|\nabla w^\eps \left(x',x_n\right)\right|^p \dd x_n \dd x', 
\end{multline*}
i.e.,
\begin{equation} \label{thin_en_ineq} 
  \mathcal{E}_0(L_\varepsilon w^\varepsilon) \leq \left(1 + \vartheta + C(\vartheta)\eps^p \left( \|\nabla ' g_+\|_{L^\infty(\omega)^n} +  \|\nabla ' g_-\|_{L^\infty(\omega)^n}\right)^p\right) \cE_\eps(w^\eps)
\end{equation} 
for any $w^\eps \in W^{1,p}(\Omega_\eps)$. Inequality \eqref{thin_en_ineq} also holds trivially for $w^\eps \in X_\eps \setminus W^{1,p}(\Omega_\varepsilon)$, because the r.h.s.\ is infinite. Then, by lower semicontinuity of $\cE_0$,  
\[  (1+ \vartheta) \liminf_{\eps \to 0^+} \cE_\eps(w^\eps) \geq \liminf_{\eps \to 0^+} \cE_0(L_\eps w
^\eps) \geq \cE_0(w).\]
Since $\vartheta$ is arbitrary, we conclude \ref{H3}. 

 The property \ref{H4} is easy to prove.
 For a given $w\in X_0$, we take $w^\varepsilon=w(x')$, a constant extension along the vertical direction $x_n$.
 By definition, we see $L_\varepsilon w^\varepsilon=w$, $\|w^\varepsilon\|_{X_\varepsilon}=\|w\|_{X_0}$, and  $\mathcal{E}_\varepsilon(w^\varepsilon)=\mathcal{E}_0(w)$, so \ref{H4} follows. 
\end{proof}

The desired convergence now follows from the abstract convergence result Theorem \ref{thm:Mosco}.
\begin{theorem} \label{Tthin}
The gradient flow $S_\varepsilon$ uniformly converges to $S_0$ along $L_\varepsilon$. \qed
\end{theorem}

If $p=1$, then \eqref{E_thin_p>1} does not define a lower semicontinuous functional on $X_\eps$. In this case, we take instead  
\begin{equation*}
    \mathcal{E}_\varepsilon(u) = \left \{
	\begin{array}{cl}
        \displaystyle \frac{1}{\varepsilon} |D u|(\Omega_\eps), & u\in L^2(\Omega_\varepsilon)\cap BV(\Omega_\varepsilon), \\
        \infty, & u\in L^2(\Omega_\varepsilon) \backslash BV (\Omega_\varepsilon).  
	\end{array}
	\right.
\end{equation*}
The gradient flow (which we still denote by $S_\eps$) of this functional is known as the total variation flow. Its description in terms of a partial differential equation can be found in \cite{ACM}. Similarly, instead of \eqref{E0_thin}, we define $\mathcal{E}_0$ by  
\begin{equation*}
    \mathcal{E}_0(v) = \left \{
	\begin{array}{cl}
        \displaystyle \int_{\omega} g(x') \; d|D' v|, & v\in L^2(\omega)\cap BV(\omega), \\
        \infty, & v\in L^2(\omega) \backslash BV (\omega), 
	\end{array}
	\right.
\end{equation*}
where $D'v$ is the derivative of $v \in BV(\omega)$. Then, $\cE_0$ is convex, lower semicontinuous and defines a gradient flow on $X_0$. Its PDE description can be found in \cite{Moll}. 

The statements of Lemma \ref{lem:Ethin} and Theorem \ref{Tthin} hold also in the case $p=1$. The only difference in the proof is that it is not enough to check \eqref{thin_en_ineq} for $w^\eps \in W^{1,1}(\Omega_\varepsilon)$, as the domain of $\cE_\eps$ is a larger space $BV(\Omega_\varepsilon)$. For $w^\eps \in BV(\Omega_\varepsilon)$, the derivative of inner composition with a bi-Lipschitz function in general is not given by the simple chain rule, and validity of the computations above is not clear. To circumvent this difficulty, we approximate $w^\varepsilon$ by a sequence $(w^\eps_j)$ of $C^1$ functions such that $w_j^\varepsilon\to w^\varepsilon$ in $L^2(\Omega_\varepsilon)$ and $\mathcal{E}_\varepsilon(w_j^\varepsilon)\to\mathcal{E}_\varepsilon(w^\varepsilon)$ as $j\to\infty$.
 Such a sequence always exists as shown in \cite{giusti}.
 With $w^\eps_j$ in place of $w^\eps$, the computations \eqref{thin_grad_calc}-\eqref{thin_en_ineq} are justified and we arrive at
\begin{equation}\label{s6-r1}
    \mathcal{E}_0(L_\varepsilon w_j^\varepsilon)
    \leq \left( 1+\vartheta+C(\vartheta)\varepsilon
    \left(\lVert \nabla' g_+\rVert_{L^\infty(\omega)^n} 
    + \lVert \nabla' g_-\rVert_{L^\infty(\omega)^n} \right)\right)
    \mathcal{E}_\varepsilon(w_j^\varepsilon).
\end{equation}
Since $L_\varepsilon w_j^\varepsilon\to L_\varepsilon w^\varepsilon$ in $L^2(\omega)$, by lower semicontinuity of $\cE_0$ we recover \eqref{thin_en_ineq} (with $p=1$) for any $w^\eps \in BV(\Omega_\varepsilon)$, and therefore for any $w^\eps \in X_\eps$. From this point on, the proof concludes as before. 

\section{The $p$-heat equation with dynamic boundary condition} \label{sec5}

\subsection{The dynamic boundary condition as a limit of boundary layer problems}\label{sec5.1}

For a given $p>1$,\footnote{The case $p=1$ is now significantly different and will be considered in Section \ref{S-TV}.} let us now consider initial value problem for the $p$-heat equation with dynamic boundary condition on a bounded  domain $\Omega$ with $C^2$ boundary
\begin{align}\label{r-dbc-plap}
&u_t  = \dv (|\nabla u|^{p-2}\nabla u) &\text{in } ]0,\infty[\times \Omega,\nonumber\\
&  u_t+  (|\nabla u|^{p-2}\nabla u) \cdot\nu =0 & \text{in } ]0, \infty[ \times \partial\Omega,\\
& u(0,\cdot) = u_0. & \nonumber
\end{align}
Our aim in this section is to show that \eqref{r-dbc-plap} arises as a limit of $p$-heat equations with boundary layer
\begin{align}\label{r-bry-plap}
&b_\eps u_t  = \dv (|\nabla u|^{p-2}\nabla u) &\text{in } ]0,\infty[\times \Omega,\nonumber\\
&  (|\nabla u|^{p-2}\nabla u) \cdot\nu =0 &\text{in } ]0,\infty[\times \partial\Omega,\\
& u(0,\cdot) = u_0, & \nonumber
\end{align}
where 
\[ b_{\eps}(x) = 1 + \eps^{-1} \mathbf 1_{\Omega \setminus \Omega^\eps}(x), \quad \Omega^\eps = \left\{ x \in \Omega \colon \dist(x, \partial \Omega) > \eps\right\}.\]

We shall first see that problems \eqref{r-bry-plap} and \eqref{r-dbc-plap} admit well-posed weak formulations as gradient flows of convex functionals on Hilbert spaces. Let $\cE_0$ be the functional defined on $X_0 :=L^2(\Omega) \times L^2(\partial\Omega)$ by
\begin{equation} \label{pheat_e0}
\cE_0(u, v)=
\left\{
\begin{array}{ll}
 \frac 1p\int_\Omega |\nabla u|^p & \text{if }u\in  W^{1,p}(\Omega),\ \gamma u = v,\\
 +\infty & \text{otherwise,}
\end{array}
\right.
\end{equation} 
where $\gamma \colon  W^{1,p}(\Omega)\to  L^{p}(\partial\Omega)$ is the trace operator. 
Next, let $X_{\eps}$ be $L^2(\Omega)$ with the scalar product 
\[(v,w)_{X_\eps} = \int_{\Omega} b_{\eps} v w.\] 
We define $\cE_\eps\colon X_\eps\to [0, \infty]$
by
\begin{equation} \label{pheat_ee}
\cE_\eps(u) = \left\{ 
\begin{array}{ll}
  \int_\Omega \frac 1p |\nabla u|^p   & u\in  W^{1,p}(\Omega),\\
 + \infty   & L^2(\Omega) \setminus W^{1,p}(\Omega),
\end{array}
\right. 
\end{equation}

\begin{proposition}\label{r-fu} Functional $\cE_\eps$ is convex and lower semicontinuous on $X_\eps$ for all $\eps \geq 0$. Moreover, $D(\cE_\eps)$ is dense in $X_\eps$ for $\eps \geq 0$.
\end{proposition}
\begin{proof} 
Convexity of $\cE_\eps$ is clear. If $\eps >0$, lower semicontinuity of $\cE_\eps$ follows from the lower semicontinuity of the $L^p$-norm. Suppose that $(u^k, v^k)\in X_0$ is a sequence converging to $(u, v)$ in $X_0$. We need to prove that 
\[\liminf_{k \to \infty} \cE(u_k, v_k) \geq \cE_0(u, v).\]
If the l.h.s\ is finite, possibly restricting to a subsequence, we may assume that $u_k$ is bounded in $W^{1,p}(\Omega)$ and $\gamma u_k = v_k$ for all $k\in \bN$. In particular, $u_k$ converges to $u$ weakly in $W^{1,p}(\Omega)$. Then, our claim follows because of the continuity of the trace operator with respect to that convergence. 

If $\eps >0$, density of $D(\cE_\eps)$ in $X_\eps$ is clear, since the norm in $X_\eps$ is equivalent to the standard norm on $L^2(\Omega)$. In order to show density of $D(\cE_0)$ in $X_0$, we first observe that $C^1(\overline{\Omega})\times C^1(\partial \Omega)$ is dense in $X_0$. For any $(u,v) \in C^1(\overline{\Omega})\times C^1(\partial \Omega)$, $k \in \bN$, we define 
\[u_k(x) := u(x) + \left(1 - k\, \dist(x, \partial \Omega)\right)_+ \left(\overline{v}(x)-u(x)\right), \quad v_k(x) := v(x),\]
where $\overline{v}$ is any extension of $v$ to $C^1(\overline{\Omega})$. Then $(u_k, v_k)$ is a sequence of elements of $D(\cE_0)$ that converges to $(u,v)$ in $X_0$.
\end{proof}


Proposition \ref{r-fu} combined with 
\cite{Brezis} 
tells us that for any $\eps \geq 0$ there exists a unique gradient flow of $\cE_\eps$ defined on the whole $X_\eps$, which we denote by $S_\eps$. In order to show that the trajectories of the flows solve \eqref{r-dbc-plap} if $\eps =0$ and \eqref{r-bry-plap} if $\eps >0$, we need to characterize the subdifferential $\partial \cE_\eps$. 
\begin{proposition}
Let $\eps >0$, $p>1$  and let $u \in D(\cE_\eps)$. Then $u \in D(\partial \cE_\eps)$ if and only if  the distributional derivative $\dv (|\nabla u|^{p-2} \nabla u)$ belongs to $L^2(\Omega)$, in which case 
\[ \partial \cE_\eps(u) = \left\{ - b_\eps^{-1} \dv (|\nabla u|^{p-2} \nabla u)\right\}. \]

Let $(u,v) \equiv (u, \gamma u) \in D(\cE_0)$. Then $(u,v) \in D(\partial \cE_0)$ if and only if  the distributional derivative $\dv (|\nabla u|^{p-2} \nabla u)$ belongs to $L^2(\Omega)$, in which case 
\[ \partial \cE_0(u,v) = \left\{ (- \dv (|\nabla u|^{p-2} \nabla u), |\nabla u|^{p-2} \nabla u \cdot \nu)\right\}. \]
\end{proposition}
\begin{proof}
We will first take care of  $\cE_\eps$ for $\eps > 0.$ 

By definition, $\xi' \in X_\eps$ is a element of $\partial \cE_\eps(u)$ if and only if for all $h\in X_\eps$ we have
\begin{equation} \label{dynp_subd}
\cE_\eps(u+h) - \cE_\eps(u) \ge\int_\Omega b_\eps \xi' h  = \int_\Omega \xi h ,
\end{equation} 
where $\xi = b_\eps \xi' \in L^2(\Omega)$. 
Reasoning as in the proof of \cite[8.2., Theorem 4]{Evans}, we check that the function $t \mapsto \cE_\eps(u + th)$ is differentiable for any $u, h \in W^{1,p}(\Omega)$, i.\,e., 
\[
\cE_\eps(u+th) - \cE_\eps(u) = t \int_\Omega |\nabla u|^{p-2}\nabla u \cdot \nabla h + o(t).
\]
Thus, choosing $\pm th$ in place of $h$ in \eqref{dynp_subd}, we get
\[ 
\pm t \int_\Omega |\nabla u|^{p-2}\nabla u \cdot \nabla h + o(t) \geq \pm \int_\Omega \xi h,
\]
whence 
\begin{equation}\label{dynp_subd2} \int_\Omega |\nabla u|^{p-2}\nabla u \cdot \nabla h = \int_\Omega \xi h.
\end{equation} 
Consequently, distributional divergence of $|\nabla u|^{p-2}\nabla u$ belongs to $L^2(\Omega)$ and we have $\xi = - \dv (|\nabla u|^{p-2}\nabla u).$ Hence, the trace $|\nabla u|^{p-2}\nabla u \cdot \nu$ exists as a functional on $\gamma (W^{1,p}(\Omega)\cap L^2(\Omega))$ defined by the equality
\[
\langle |\nabla u|^{p-2}\nabla u \cdot\nu, \gamma h\rangle = \int_\Omega |\nabla u|^{p-2}\nabla u \cdot \nabla h + \int_\Omega \dv (|\nabla u|^{p-2}\nabla u) h.  
\]
Since \eqref{dynp_subd2} holds for any $h \in W^{1,p}(\Omega)$, $(|\nabla u|^{p-2}\nabla u)\cdot \nu =0$. On the other hand, suppose that $\xi = - \dv (|\nabla u|^{p-2}\nabla u)$ and $|\nabla u|^{p-2}\nabla u \cdot\nu = 0$. Then, by convexity of the function $v \mapsto |v|^p$, we have
\[
\cE_\eps(u+h) - \cE_\eps(u) \ge\int_\Omega |\nabla u|^{p-2}\nabla u \cdot \nabla h = \int_\Omega \xi h.
\] 
This concludes the characterization of $\partial\cE_\eps$ for $\eps>0.$

In case $\eps=0$ we proceed in a similar manner. A pair 
$(\xi_1, \xi_2) \in X_0$ belongs to $\partial\cE_0(u,v)$ if and only if
\begin{equation} \label{dynp_subd3}
\cE_0(u+h_1, v+h_2) - \cE_0(u, v) \ge \int_\Omega \xi_1 h_1 + \int_{\partial\Omega} \xi_2 h_2
\end{equation}
for all $(h_1, h_2) \in X_0$. The LHS is finite when $h_1\in W^{1,p}(\Omega)$ and $\gamma h_1 = h_2.$ In this case, arguing as above we reach the identity
\[
\int_\Omega |\nabla u |^{p-2} \nabla u \nabla h_1 = \int_\Omega \xi_1 h_1 + \int_{\partial\Omega } \xi_2 h_2
\]
for all $h_1\in W^{1,p}(\Omega)$, $h_2\in L^2(\partial\Omega)$ such that $\gamma h_1 = h_2.$ This implies that
\begin{equation}\label{dynp_subd4} 
\xi_1 = -\dv (|\nabla u|^{p-2}\nabla u)\qquad\text{and}\qquad \xi_2 = |\nabla u|^{p-2}\nabla u\cdot \nu.
\end{equation} 
On the other hand, if $(\xi_1, \xi_2) \in X_0$ are as in \eqref{dynp_subd4}, then \eqref{dynp_subd3} holds, which concludes the characterization of $\partial \cE_0$.
\end{proof}

In order to discuss convergence of gradient flows $S_\eps$, we need to define suitable connecting operators. We have to make some preparations. Let us define $\Phi^\Omega\colon\partial\Omega\times ]\!-\!1,1[\to \bR^N$ by the following formula,
\begin{equation} \label{def_Phi} 
\Phi^\Omega(x', x_n) = x'- \nu(x') x_n,
\end{equation} 
where $\nu(x')$ is the outer normal vector to $\partial\Omega$ at $x'$. In non-ambiguous cases we will write shortly $\Phi$. 

Clearly, for every $x'\in \partial\Omega$ the derivative  $D\Phi(x', 0)$ is non-singular. Hence, there is such $\eps_0>0$ that $\Phi$ restricted to $\partial\Omega\times ]-\eps_0,\eps_0[$ is a diffeomorphism.  Let us notice that for all $\eps\in ]0,\eps_0[$ we have 
\begin{equation}\label{radius}
\Phi(\partial\Omega \times \{\eps\}) = \partial \Omega^\eps.
\end{equation}
Indeed, suppose that $y \in \partial \Omega$. It is well known that there exists a nearest point projection $x$ of $y$ onto $\partial \Omega$, and the vector $x-y$ is perpendicular to $\partial \Omega$. More precisely, we have $x-y = \eps \nu(x)$, so 
\[ y = x - \eps \nu(x) = \Phi(x, \eps).\]
On the other hand, suppose that $y = \Phi(x, \eps)$ with $x \in \partial \Omega$. Let us denote $\delta:=\dist(y, \partial \Omega^\eps$). Clearly $\delta \leq \eps$. If the inequality were strict, there would exist a nearest point projection $z$ of $y$ onto $\partial \Omega$ with 
\[\Phi(z, \delta) = z - \delta \nu (z) = y = \Phi(x,\eps),\]
violating injectivity of $\Phi$. Thus, $\delta = \eps$, whence \eqref{radius} follows. The same argument also shows that the nearest point projection $\pi(x)$ of $x \in \partial \Omega^\eps$ onto $\partial \Omega$ is unique for $\eps \in [0, \eps_0[$, and
\begin{equation} \label{pi_phi} 
\Phi(\pi(x), \eps) = x \quad \text{for } x \in \partial \Omega^\eps,\ \eps \in [0, \eps_0[. 
\end{equation}

Since $\Phi$ is smooth, formulae
\begin{equation} \label{def_m_L} 
m_\eps w^\eps( x') = \frac{1}{\eps} \int_0^\eps  w^\eps(\Phi(x',x_n))\dd x_n, \quad L_\eps w^\eps = (w^\eps, m_\eps w^\eps)
\end{equation} 
define bounded linear operators $m_\eps = m^\Omega_\eps \colon X_\eps \to L^2(\partial \Omega)$, $L_\eps \colon X_\eps \to X_0$.

\begin{lemma}\label{lem:pheatX}
$X_\eps$ Mosco-converge to $X_0$ along $L_\eps$. 
\end{lemma} 

\begin{proof} 
We first show asymptotic contractivity of $L_\eps$.  For $w^\eps \in X_\eps$ we have 
\[\|L_\eps w^\eps\|_{X_0}^2 = \|w^\eps\|_{L^2(\Omega)}^2 + \|m_\eps w^\eps\|_{L^2(\partial \Omega)}^2 .
\]
Now, we have to evaluate the last term on the r.h.s.,
\[
\|m_\eps w^\eps\|_{L^2(\partial \Omega)}^2\le \frac 1\eps\int_{\partial\Omega}\int_0^\eps w^2(\Phi(x', x_n)) \, \dd x_n \dd\cH^{n-1}(x') = 
 \frac 1\eps\int_{\Omega\setminus\overline{\Omega}^\eps} w^2 (y)J^{-1}(y)\,\dd y,
\]
where $J^{-1}(y) = |\det D\Phi^{-1}(y)|$.  This observation leads us to
\[
\|L_\eps w^\eps\|_{X_0}^2 
\leq \int_\Omega \left(1 + J^{-1} \tfrac{1}{\eps} \mathbf 1_{\Omega\setminus\Omega^\eps}
\right) (w^\eps)^2 \le \max_{\overline \Omega\setminus\Omega^\eps}J^{-1}\int_{\Omega}
b_\eps (w^\eps)^2 
= \max_{\overline \Omega\setminus\Omega^\eps}J^{-1}\|w^\eps\|_{X_\eps}^2 , 
\]
where we have used (\ref{radius}). Since $J^{-1}=1$ on $\partial\Omega$ (and $J^{-1}$ is continuous), this implies \ref{H1}.

Next, given $w = (w_1,w_2) \in X_0$, we take $w^\eps \in X_\eps$ defined by 
\[\label{w_dbc_approx} 
 w^\eps(x) = 
\left\{\begin{array}{ll}w_1(x) & \text{if } x\in {\Omega^{\eps}}, \\ 
w_2(\pi  (x)) & \text{if } 
x \in \Omega\setminus\overline \Omega^{\eps}.\end{array}\right.
\]
Then,
\[\|w^\eps\|_{X_\eps}^2 = \int_{\Omega^\eps} w_1^2\, \dd x + \int_{ \Omega\setminus\Omega^\eps} (1+ \frac1\eps) w_2^2(\pi(x))\,\dd x.\] 
After a change of variables $x = \Phi(y)$, using \eqref{pi_phi}, we get
\[
\frac1\eps\int_{ \Omega\setminus\Omega_\eps}  w_2^2(\pi(x))\,\dd x
= \int_{\partial\Omega} w_2^2(y') \frac1\eps\int_0^\eps K(y',y_n)\,\dd y_n \dd \cH^{n-1}(y') \to \int_{\partial\Omega} w_2^2\, \dd \cH^{n-1},
\]
because $K: = |\det D\Phi|$ satisfies
\begin{equation}\label{r-J}
K(y',0) = 1 \quad \text{for } y' \in \partial \Omega. 
\end{equation} 
Hence,
\[ \|w^\eps\|_{X_\eps}^2 \to  \|w\|_{X_0}^2 \]
and 
\[L_\eps w^\eps = (w^\eps, w_2) \to (w_1, w_2) = w \quad \text{in } X_0.\]
Thus, (H2) holds.
\end{proof} 

\begin{lemma}\label{lem:pheatE}
$\cE_\eps$ Mosco-converge to $\cE_0$ along $L_\eps$. 
\end{lemma} 

\begin{proof} 
First, we check condition \ref{H3}. Let us take a sequence $w^\eps\in X_\eps$ such that 
\[
\limsup_{\eps\to 0} \| w^\eps\|_{X_\eps} <\infty\qquad\hbox{and}\qquad
L_\eps w^\eps \rightharpoonup (u, v).
\]
We also suppose that 
\[\liminf_{\eps\to 0} \cE_\eps(w^\eps) < \infty\] 
(otherwise there is nothing to prove). Then $u \in W^{1,p}(\Omega)$ and, by the lower semicontinuity of the $L^p$ norm,
\[
\liminf_{\eps\to 0} \cE_\eps(w^\eps)  = \liminf_{\eps\to 0} \frac 1p \int_\Omega |\nabla w^\eps|^p \geq    \frac 1p \int_\Omega |\nabla u|^p  .
\]
Next we carry out estimates near the boundary. By Young's inequality, for $\cH^{n-1}$-a.\,e.\ $x' \in \partial \Omega$,
\begin{multline*}  |m_\eps w^\eps(x') - \gamma w^\eps(x')|^p = \left|\frac{1}{\eps}\int_0^\eps w^\eps(\Phi(x', s)) - w^\eps(\Phi(x', 0))\dd s\right|^p \\ \leq \frac{1}{\eps}\int_0^\eps | w^\eps(\Phi(x', s)) - w^\eps(\Phi(x', 0))|^p\dd s  = \frac{1}{\eps}\int_0^\eps \left| \int_0^s \nabla w^\eps(\Phi(x',  \sigma)) \cdot \Phi_{x_n}(x', \sigma)) \dd \sigma\right|^p\dd s.    
\end{multline*} 
Taking into account that $|\Phi_{x_n}(x', \sigma)| = |\nu(x')|=1$, we get using H\" older's inequality
\[|m_\eps w^\eps - \gamma w^\eps|^p \leq \frac{1}{\eps}\int_0^\eps  s^{p-1}\int_0^s \left|\nabla w^\eps(\Phi(x',  \sigma))\right|^p  \dd \sigma\dd s  \leq  \eps^{p-1}\int_0^\eps \left|\nabla w^\eps(\Phi(x', s))\right|^p  \dd s.\]
Thus, changing variables as before, 
\begin{multline} \label{mtrace_est}
\| m_\eps w^\eps - \gamma w^\eps\|_{L^p(\partial\Omega)}^p \leq  \eps^{p-1}\int_{\partial \Omega}\int_0^\eps \left|\nabla w^\eps(\Phi(x', s))\right|^p  \dd s \,\dd \cH^{n-1}\\= \eps^{p-1}\int_{ \Omega \setminus \Omega^\eps}\left|\nabla w^\eps(\Phi(y))\right|^p J^{-1}(y) \,\dd y \leq \eps^{p-1} \max_{\overline \Omega\setminus\Omega^\eps}J^{-1} \| \nabla w^\eps\|_{L^p(\Omega)}^p
\end{multline}

Let $(w^{\eps_k})$ be a subsequence of $(w^\eps)$ such that $\|\nabla w^{\eps_k}\|_{L^p(\Omega)}$ is bounded and $w^{\eps_k} \rightharpoonup u$ in $W^{1,p}(\Omega)$. The r.h.s.\ of \eqref{mtrace_est} converges to $0$ along that subsequence. On the other hand, $\gamma w^{\eps_k} \rightharpoonup \gamma u$. 
Hence, $\gamma u =v$ and so 
\[
\liminf_{\eps\to 0^+} \cE_\eps(w^\eps) \geq \frac 1p \int_\Omega |\nabla u|^p = \cE_0(u,v),
\]
i.e.\ \ref{H3} follows. 

It remains to establish \ref{H4}. Let $(u,v)\in D(\cE_0)$, whence $v = \gamma u$. We take $w^\eps = u$. Then, $L_\eps w^\eps = (u, m_\eps u )$. Of course $m_\eps u$ tends to $\gamma u =v$. Hence,
\[
\lim_{\eps\to 0^+} \cE_\eps(w^\eps) = \frac 1p \int_\Omega |\nabla u|^p = \cE_0(u,v) 
\]
and \ref{H4} follows.
\end{proof} 

Appealing to Theorem \ref{thm:Mosco}, we deduce
\begin{theorem} \label{thm:pflow_conv}
    $S_\eps$ uniformly converge to $S_0$ along $L_\eps$. \qed 
\end{theorem}

We note that Theorem \ref{thm:pflow_conv} extends \cite[Theorem 3.1]{glr} to the array of the $p$-heat equation with $p>1$.

\subsection{Limiting behavior of the flow with dynamic boundary condtion }\label{sec5.2}
Now, let us consider the following variant of \eqref{r-dbc-plap} with parameter $\tau >0$:
\begin{align}\label{r-dbc-tau}
&u_t  = \dv (|\nabla u|^{p-2}\nabla u) &\text{in } ]0,\infty[\times \Omega,\nonumber\\
& \tau u_t+  (|\nabla u|^{p-2}\nabla u) \cdot\nu =0 & \text{in } ]0, \infty[ \times \partial\Omega,\\
& u(0,\cdot) = u_0. & \nonumber
\end{align}
Let $X_\tau$ be the Hilbert space $L^2(\Omega) \times L^2(\partial\Omega)$ with the scalar product 
\[(w,v)_{X_\tau} = \int_\Omega w_1 v_1 + \tau \int_{\partial \Omega} w_2 v_2 . \]
As before, we can show that the system \eqref{r-dbc-tau} coincides with the gradient flow of $\cE_\tau \colon X_\tau \to [0, \infty]$ given by 
\[
\cE_\tau(u, v)=
\left\{
\begin{array}{ll}
 \frac 1p\int_\Omega |\nabla u|^p & \text{if }u\in  W^{1,p}(\Omega),\ \gamma u = v,\\
 \infty & \text{otherwise.}
\end{array}
\right.
\]
Moreover, it is easy to deduce from Theorem \ref{thm:pflow_conv} that the system \eqref{r-dbc-tau} arises as a limit of boundary layer problems with weight $b_{\eps,\tau}(x) = 1 + \tau \eps^{-1} \mathbf 1_{\Omega \setminus \Omega^\eps}(x)$ as $\eps \to 0^+$. 

We now wish to investigate the limiting behavior of solutions to \eqref{r-dbc-tau} as $\tau \to 0^+$ and as $\tau \to \infty$. We expect the limits to solve the homogeneous Neumann (respectively, Dirichlet) problem for the $p$-heat equation: 
\begin{align}\label{r-Neu}
&u_t  = \dv (|\nabla u|^{p-2}\nabla u) &\text{in } ]0,\infty[\times \Omega,\nonumber\\
& \nabla u \cdot\nu =0 & \text{in } ]0, \infty[ \times \partial\Omega,\\
& u(0,\cdot) = u_0, & \text{in } \Omega; \nonumber
\end{align}
(respectively,
\begin{align}\label{r-Dir}
&u_t  = \dv (|\nabla u|^{p-2}\nabla u) &\text{in } ]0,\infty[\times \Omega,\nonumber\\
& u =0 & \text{in } ]0, \infty[ \times \partial\Omega,\\
& u(0,\cdot) = u_0, & \text{in } \Omega). \nonumber
\end{align}
Let us denote $X_0 = X_\infty : = L^2(\Omega)$ and define $\cE_0 \colon X_0 \to [0, \infty]$, as well as $\cE_\infty \colon X_\infty \to [0, \infty]$ by 
\[ \cE_0 = \left\{ 
\begin{array}{ll}
  \int_\Omega \frac 1p |\nabla u|^p   & u\in  W^{1,p}(\Omega),\\
 + \infty   & L^2(\Omega) \setminus W^{1,p}(\Omega),
\end{array} \right. \quad \cE_\infty = \left\{ 
\begin{array}{ll}
  \int_\Omega \frac 1p |\nabla u|^p   & u\in  W^{1,p}_0(\Omega),\\
 + \infty   & L^2(\Omega) \setminus W^{1,p}_0(\Omega).
\end{array} \right. \]
It is well known that \eqref{r-Neu} and \eqref{r-Dir} admit weak formulations as gradient flows of $\cE_0$ and $\cE_\infty$, respectively. We further denote $L_\tau  (u^\tau, v^\tau) = u^\tau$ for $(u^\tau, v^\tau)$ in $X_\tau$. We first prove 

\begin{theorem} \label{thm:pflow_tau_conv}
    $S_\tau$ uniformly converges to $S_0$ along $L_\tau$ as $\tau \to 0^+$. 
\end{theorem} 
\begin{proof} 
Clearly assumption \ref{H1} holds. As for \ref{H2}, given $u \in X_0$, it is enough to take $(u^\tau, v^\tau) = (u, 0)$. Next, taking into account that 
\[\cE_\tau ((u^\tau, v^\tau)) \geq \cE_0 (u^\tau) =  \cE_0(L_\eps(u^\tau, v^\tau))\]
for any $(u^\tau, v^\tau) \in X_\tau$, hypothesis \ref{H3} follows by lower semicontinuity and convexity of $\cE_0$ as before. It remains to find a recovery sequence for any $u \in D(\cE_0)$, i.e.,  $u \in W^{1,p}(\Omega)$. We can define $(u^\tau, v^\tau) := (u, \gamma\, u)$ for $\tau >0$, whence
\[\|(u^\tau, v^\tau)\|_{X_\eps}^2 = \|u\|_{L^2 (\Omega)}^2 + \tau \|\gamma\, u \|_{L^2(\partial \Omega)}^2 \to \|u\|_{L^2 (\Omega)}^2 \quad \text{as } \tau \to 0^+,\] 
\[L_\tau (u^\tau, v^\tau) = u,\quad \cE_\tau ((u^\tau, v^\tau)) = \cE_0(u) \quad  \text{for } \tau > 0.\]
This shows the required Mosco-convergence of $\cE_\tau$. An application of Theorem \ref{thm:Mosco} concludes the proof. 
\end{proof} 

\begin{theorem} \label{thm:pflow_tau_convD}
    $S_\tau$ uniformly converges to $S_\infty$ along $L_\tau$ as $\tau \to \infty$. 
\end{theorem} 

\begin{proof} 
Again, hypothesis \ref{H1} is obvious and, for any $u \in X_\infty$, we have $L_\tau (u,0) = u$, so \ref{H2} holds. Let us move on to Mosco-convergence of $\cE_\tau$. Let  $(u^\tau, v^\tau)$ be a sequence converging weakly to $u$ along $L_\tau$. By definition, we have 
\[ \limsup_{\tau \to \infty} \|u^\tau\|_{L^2(\Omega)}^2 + \tau\|v^\tau\|_{L^2(\partial \Omega)}^2 < \infty.\]
In particular, $v^\tau \to 0$ in $L^2(\partial \Omega)$ as $\tau \to \infty$. Suppose that 
\begin{equation} \label{pflow_tau_as_bd} \limsup_{k \to \infty} \cE_\tau((u^{\tau_k} ,v^{\tau_k})) < \infty
\end{equation} 
for a subsequence $t_k \to \infty$. Then $(u^{\tau_k})$ is asymptotically bounded in $W^{1,p}(\Omega)$ and $u^{\tau_k} \rightharpoonup u$ in $W^{1,p}(\Omega)$. Thus, $\gamma u^{\tau_k} \rightharpoonup \gamma u$ in $L^p(\partial \Omega)$. On the other hand, also by \eqref{pflow_tau_as_bd}, $\gamma u^{\tau_k} = v^{\tau_k}$ for large $k$. Since we already observed that $v^\tau \to 0$, we have $\gamma u =0 $, i.e., $u \in W^{1,p}_0(\Omega)$. From this and lower semicontinuity of the $L^p$ norm, we deduce that \ref{H3} holds. 

Finally, in order to show \ref{H4}, we take $u \in L^2(\Omega) \cap W^{1,p}_0(\Omega)$. For $\tau>0$ we take $(u^\tau, v^\tau):= (u, \gamma u) = (u,0)$. It is easy to check that it is a good recovery sequence. 

We conclude by Theorem \ref{thm:Mosco} as usual. 
\end{proof}

\section{Total variation flow with dynamic boundary conditions}\label{S-TV}

\subsection{Dynamic boundary conditions as a limit of boundary layer problems} \label{sS-TV-lay-dbc} 

In the case $p=1$, the proper notion of solution to systems \eqref{r-dbc-plap} and \eqref{r-bry-plap} is more involved. Functionals $\cE_0$ and $\cE_\eps$ given by \eqref{pheat_e0} and \eqref{pheat_ee} are not lower semicontinuous. Instead, we take their lower semicontinuous envelopes, given by 
\begin{equation} \label{tvf_e0}
\cE_0(u, v)=
\left\{
\begin{array}{ll}
 |Du|(\Omega) + \int_{\partial \Omega} |\gamma\, u - v| & \text{if }u\in  BV(\Omega),\\
 +\infty & \text{otherwise,}
\end{array}
\right.
\end{equation} 
where $\gamma \colon  BV(\Omega)\to  L^1(\partial\Omega)$ is the trace operator, and 
\begin{equation} \label{tvf_ee}
\cE_\eps(u) = \left\{ 
\begin{array}{ll}
  |D u|(\Omega)   & u\in  BV(\Omega),\\
 + \infty   & \text{otherwise.}
\end{array}
\right. 
\end{equation} 
The functionals are defined on the spaces $X_0$, $X_\eps$ respectively, which are defined as in the previous section: $X_0=L^2(\Omega)\times L^2(\partial \Omega)$ with the standard inner product, $X_\eps = L^2(\Omega)$, $(w,v)_{X_{\eps}} = \int_\Omega b_\eps wv$. Also as in the previous section, we assume that $\Omega$ is a bounded domain with $C^2$ boundary. The subdifferential of $\cE_0$ was characterized in \cite[Theorem 5.1]{gnrs}, leading to the following description of its gradient flow $S_0$:
\begin{equation} \label{tvfdbcbulk}
u_t = \dv z \quad \text{in } \Omega_T, 
\end{equation} 
\begin{equation} \label{tvfdbcbdry}  
f_t = - z \cdot \nu^\Omega \quad \text{in } \partial \Omega_T 
\end{equation} 
with $z \in L^\infty(\Omega_T)$ satisfying $\dv z \in L^2(\Omega_T)$ and 
\begin{equation}\label{tvfdbcz1} 
|z|\leq 1 \quad \text{in } \Omega_T, 
\end{equation} 
\begin{equation}\label{tvfdbczanz}
(z, D u) = |D u|,  
\end{equation} 
\begin{equation}\label{tvfdbczbdry} 
z \cdot \nu^\Omega \in \sgn (f-\gamma u) \quad \text{on } \partial \Omega_T. 
\end{equation}
The symbol $(z, Du)$ appearing in \eqref{tvfdbczanz} denotes the Anzellotti pairing, see \cite{anzellotti}. 

We want to justify that the gradient flow $S_0$ arises as the limit of gradient flows $S_\eps$ of $\cE_\eps$, corresponding to the system
\begin{equation} \label{tvfblbulk}
b_\eps u^\eps_t = \dv z^\eps \quad \text{in } \Omega_T, 
\end{equation} 
\begin{equation} \label{tvfblbdry}  
z^\eps \cdot \nu^\Omega =0\quad \text{in } \partial \Omega_T 
\end{equation} 
with $z^\eps \in L^\infty(\Omega_T)$ satisfying $\dv z^\eps \in L^2(\Omega_T)$ and 
\begin{equation}\label{tvfblz1} 
|z^\eps|\leq 1 \quad \text{in } \Omega_T, 
\end{equation} 
\begin{equation}\label{tvfblzanz}
(z^\eps, D u^\eps) = |D u^\eps|.  
\end{equation} 
In \eqref{tvfblbulk}, $b_\eps = \mathbf 1_\Omega + \frac{1}{\eps} \mathbf 1_{\Omega\setminus \Omega_\eps}$, where $\Omega_\eps = \{x\in \Omega: \dist(x, \partial\Omega)> \eps\}$.

We recall operators $m_\eps \colon X_\eps \to L^2(\partial \Omega)$, 
$L_\eps \colon X_\eps \to X_0$ defined by \eqref{def_m_L}. Convergence of $X_\eps$ to $X_0$ along $L_\eps$ has already been shown in Lemma \ref{lem:pheatX}. In order to apply Theorem \ref{thm:Mosco}, it remains to prove 

\begin{lemma} \label{lem:tvfE}
$\cE_\eps$ Mosco-converges to $\cE_0$ along $L_\eps$.
\end{lemma} 

\begin{proof} 
Let $(w^\eps) \subset X_\eps$  be such that 
$\limsup \|w^\eps\|_{X_\eps} < \infty$ and $L_\eps w^\eps \rightharpoonup (w_1,w_2)$ in $X_0$. Fix $\delta>0$ small enough.  In order to derive a necessary estimate, we first assume that $w^\eps \in C^1(\Omega)$. For $0<a<\eps<\delta<b<2\delta$, we estimate
\begin{multline} \label{r-trace}
\int_{\partial\Omega} | w^\eps(\Phi(x', b)) - w^\eps(\Phi(x', a)) |\,\dd\cH^{n-1}(x') = 
\int_{\partial\Omega} \left| \int_{a}^{b} \frac\partial{\partial s} w^\eps(\Phi(x',s))\,\dd s \right|\,\dd\cH^{n-1}(x')\\
\le \int_{\partial\Omega}\int_a^b | \nabla w^\eps(\Phi(x', s))|\,
\dd s\, \dd\cH^{n-1}(x')
\le \int_{\Omega \setminus \Omega^{2 \delta}} |\nabla w^\eps(x)| J^{-1}(x)\, \dd\cL^n(x).
\end{multline}
In the first inequality we used $|\frac\partial{\partial x_n} \Phi| = |\nu|= 1$, in the second inequality we changed the variables. 

Next, for $\delta>0$ small enough, $w \in L^1(\Omega)$, and $x' \in \partial \Omega$ we denote 
\[ m_{\delta, 2 \delta}\, w (x') = \int_\delta^{2 \delta} w(\Phi(x', b)) \,\dd b. \]
Applying the triangle inequality
and \eqref{r-trace}, we get 
\begin{align} 
\nonumber \int_{\partial \Omega}& |m_{\delta, 2 \delta}\, w^\eps(x') -  m_\eps\, w^\eps(x')| \,\dd\cH^{n-1}(x')
\\ \nonumber & = \int_{\partial \Omega}  \left|\frac{1}{\delta} \int_\delta^{2 \delta} \frac{1}{\eps}\int_0^\eps  w^\eps(\Phi(x',b)) - w^\eps(\Phi(x', a)) \,\dd a \, \dd b \right| \dd\cH^{n-1}(x')\\
\label{m_est_calc} & \leq \frac{1}{\delta} \int_\delta^{2 \delta} \frac{1}{\eps}\int_0^\eps  \int_{\partial \Omega}\left| w^\eps(\Phi(x', b)) - w^\eps(\Phi(x', a))\right|\dd\cH^{n-1}(x')\, \dd a\, \dd b \\ \nonumber &
\le \frac{1}{\delta} \int_\delta^{2 \delta} \frac{1}{\eps}\int_0^\eps  \int_{\Omega \setminus \Omega^{2 \delta}} |\nabla w^\eps(x)| J^{-1}(x)\, \dd\cL^n(x) \, \dd a\, \dd b \\ \nonumber &= \int_{\Omega \setminus \Omega^{2 \delta}} |\nabla w^\eps(x)| J^{-1}(x)\, \dd\cL^n(x) \leq \max_{\overline{\Omega} \setminus \Omega^{2 \delta}} J^{-1}\ \int_{\Omega \setminus \Omega^{2 \delta}} |\nabla w^\eps|\, \dd\cL^n.
\end{align}
Now, for any $w^\eps \in BV(\Omega)$ we can take a sequence of smooth functions $(w^{\eps}_k)$ that approximates $w^\eps$ in the strict sense, see \cite[Remark 3.22]{AFP}. Then 
\begin{multline*} \limsup_{k \to \infty} \int_{\Omega \setminus \Omega^{2 \delta}} |\nabla w^\eps_k|\, \dd\cL^n = \lim_{k \to \infty} \int_{\Omega} |\nabla w^\eps_k|\, \dd\cL^n - \liminf_{k \to \infty} \int_{\Omega^{2 \delta}} |\nabla w^\eps_k|\, \dd\cL^n \\ \leq |D w^\eps|(\Omega) - |D w^\eps|(\Omega^{2 \delta}) = |D w^\eps|(\Omega \setminus \Omega^{2 \delta}).
\end{multline*}
Thus, passing to the limit $k \to \infty$ in \eqref{m_est_calc} with $w^\eps_k$ in place of $w^\eps$, we obtain
\begin{equation} \label{m_est}
\int_{\partial \Omega} |m_{\delta, 2 \delta}\, w^\eps(x') -  m_\eps\, w^\eps(x')| \,\dd\cH^{n-1}(x') \leq \max_{\overline{\Omega} \setminus \Omega^{2 \delta}} J^{-1}\ |D w^\eps|(\Omega \setminus \Omega^{2 \delta})
\end{equation} 
for $\eps \in ]0, \delta[$. Therefore, since $\max_{\overline{\Omega} \setminus \Omega^{2 \delta}} J^{-1} = \big(\min_{\overline{\Omega} \setminus \Omega^{2 \delta}} J\big)^{-1}$, we have for any $w^\eps \in D(\cE_\eps)$
\begin{equation}
\cE_\eps(w^\eps) \geq  |D w^\eps|(\Omega^{2\delta}) + 
\min_{\overline{\Omega} \setminus \Omega^{2 \delta}} J 
\int_{\partial \Omega} |m_{\delta, 2 \delta}\, w^\eps -  m_\eps\, w^\eps|. 
\end{equation}

By the weak convergence $(w^\eps, m_\eps\, w^\eps) = L_\eps w^\eps \rightharpoonup (w_1, w_2)$ in $X_0$, convexity and lower semicontinuity of the $BV$ and $L^1$ norms we obtain 
\begin{equation} \label{liminf_tv_bl}
\liminf_{\eps\to 0^+}\cE_\eps(w^\eps) \geq |D w_1|(\Omega^{2\delta}) + \min_{\overline{\Omega} \setminus \Omega^{2 \delta}} J \ \int_{\partial \Omega} |m_{\delta, 2 \delta}\, w_1 -  w_2|.
\end{equation}
Here, we also used that, by continuity of the operator $m_{\delta,2 \delta}\colon L^2(\Omega) \to L^2(\partial \Omega)$, if $w^\eps \rightharpoonup w_1 $ in $L^2(\Omega)$, then $m_{\delta,2 \delta}\, w^\eps \rightharpoonup m_{\delta,2 \delta}\, w_1$ in $L^2(\partial\Omega)$.

It remains to pass to the limit $\delta \to 0^+$. We observe that $m_{\delta, 2 \delta}\, w_1 \to \gamma\, w_1$ in $L^1(\partial \Omega)$. This can be ascertained by inequality   
\begin{equation*} 
\int_{\partial \Omega} |m_{\delta, 2 \delta}\, w_1(x') -  \gamma\, w_1(x')| \,\dd\cH^{n-1}(x') \leq \max_{\overline{\Omega} \setminus \Omega^{2 \delta}} J^{-1}\, |D w_1(x)|(\Omega \setminus \Omega^{2 \delta}), 
\end{equation*} 
which can be proved in a manner analogous to \eqref{m_est}, see also \eqref{mtrace_est}. (We note that due to the smoothness of $\partial \Omega$, we can take a strict approximation of $w^\eps$ smooth up to the boundary.) 
Taking into account that, by (\ref{r-J}), $\lim_{\delta\to 0^+} \min_{\overline{\Omega} \setminus \Omega^{2 \delta}} J=1$, we recover from \eqref{liminf_tv_bl} the assertion of \ref{H3}. 

Finally, we check \ref{H4} for a given $w = (w_1, w_2) \in X_0$ with $w_1 \in BV(\Omega)$. Let $(\widetilde{w}_2^k)$ be a sequence of smooth functions on $\partial \Omega$ converging to $w_2$ in $L^2$. For small enough $\eps >0$, we set 
\[ w^\eps =  \mathbf{1}_{\Omega^\eps}\, w_1 + \mathbf{1}_{\Omega \setminus \Omega^\eps}\, \widetilde{w}_2^{k(\eps)} \circ \pi, \]
where $\pi$ is the nearest point projection onto $\partial \Omega$ and $k(\eps)$ is a function satisfying $k(\eps) \to \infty$ as $\eps \to 0^+$, that will be chosen later. By \cite[Corollary 3.89]{AFP}, $w^\eps \in BV(\Omega)$ and 
\[ |D w^\eps|(\Omega) = |Dw_1|(\Omega^\eps) + \int_{\partial \Omega^\eps} |\gamma^{\Omega^\eps} w_1 - \widetilde{w}_2^{k(\eps)} \circ \pi|\,\dd \cH^{n-1} + \int_{\Omega \setminus \Omega^\eps} |\nabla (\widetilde{w}_2^{k(\eps)} \circ \pi)|\, \dd \cL^n =: I_1^\eps + I_2^\eps + I_3^\eps. \]
We estimate, using a change of variables and appealing to smoothness of $\Omega$,  
\begin{multline*} I_3^\eps \leq  \int_{\Omega \setminus \Omega^\eps} |\nabla \widetilde{w}_2^{k(\eps)}(\pi(x))| \, | D \pi(x)|\, \dd x \\ =   \int_0^\eps \int_{\partial \Omega} |\nabla \widetilde{w}_2^{k(\eps)}(x')| \, | D \pi(\Phi(x',x_n))|\, K(x', x_n)\dd \cH^{n-1}(x') \dd x_n \leq C \eps \int_{\partial \Omega} |\nabla \widetilde{w}_2^{k(\eps)}|\dd \cH^{n-1}.
\end{multline*} 
Thus, choosing $k(\eps)$ converging to $\infty$ slowly enough, we get $I_3^\eps \to 0$ as $\eps \to 0^+$. 

Next, by the change of variables $y' = \Phi_\eps(x') := \Phi(x', \eps)$, appealing to \eqref{pi_phi}, we see that
\[I_2^\eps = \int_{\partial \Omega} |\gamma^{\Omega} w_1 \circ \Phi_\eps - \widetilde{w}_2^{k(\eps)} |\,J_\eps^{-1}\dd \cH^{n-1},\]
where $J_\eps^{-1}:= |\det D(\Phi_\eps^{-1})|$. Since $\gamma^{\Omega} w_1 \circ \Phi_\eps \to \gamma^{\Omega} w_1$ in $L^1(\partial \Omega)$ (this again follows from an inequality similar to \eqref{m_est}) and $J_\eps^{-1} \to 1$ uniformly as $\eps \to 0^+$, we have 
\[ \lim_{\eps \to 0^+} I_2^\eps = 
\int_{\partial \Omega} |\gamma^{\Omega} w_1 - w_2 |\dd \cH^{n-1}. \]
Finally, the convergence $I_1^\eps \to |Dw_1|(\Omega)$ is clear. Summing up, we have 
\[\cE_\eps(w^\eps) = |Dw^\eps|(\Omega) \to |Dw_1|(\Omega) + \int_{\partial \Omega} |\gamma^{\Omega} w_1 - w_2 |\dd \cH^{n-1} = \cE_0(w_1,w_2)\]
as $\eps \to 0^+$, which concludes the proof of \ref{H4}.
\end{proof} 

\begin{theorem} \label{thm7.1}
    $S_\eps$ converges uniformly to $S_0$ along $L_\eps$. \qed
\end{theorem} 

\subsection{Neumann boundary condition as a limit of dynamic b.\,c.} 

Let $X_\tau$ be the space $L^2(\Omega)\times L^2(\partial \Omega)$ with scalar product 
\[(w,v)_{X_\tau} = \int_\Omega w_1 v_1 + \tau \int_{\partial \Omega} w_2 v_2 . \]
The total variation flow with dynamic b.\,c., $S_\tau$, is the gradient flow  of the functional $\cE_\tau\colon X_\tau \to [0,\infty],$ defined by 
\[\cE_\tau(w_1,w_2)=\int_\Omega |\nabla w_1| + \int_{\partial \Omega} |\gamma w_1- w_2|.\]

We recall that the usual total variation flow with Neumann boundary condition $S_0$ is the gradient flow of $\cE_0 \colon X_0 \to [0, \infty],$ given by 
\[\cE_0(w)=\int_\Omega |D w|,\]
where $X_0 = L^2(\Omega)$. 

Let $L_\tau \colon X_\tau \to X_0$ be given by 
\[L_\tau (w^\tau_1, w^\tau_2) = w^\tau_1 . \]

\begin{theorem} 
$S_\tau$ uniformly
converge to $S_0$ along $L_\tau$ as $\tau \to 0^+$. 
\end{theorem} 

\begin{proof} 
Mosco convergence of $X_\tau$ as $\tau \to 0^+$ was shown in \ref{thm:pflow_tau_conv}.  We move on to checking the Mosco convergence of $\cE_\tau$ to $\cE_0$ along $L_\tau$. 
The inequality \ref{H3} follows easily from lower semicontinuity of total variation. It remains to check \ref{H4}. We take $w \in X_0$ such that $\cE_0(w) < \infty$. Then the boundary trace $\gamma w \in L^1(\partial \Omega)$ is well defined. Let $k(\tau)$ be such that $k(\tau) \to \infty$ as $\tau \to 0^+$, and let $T_{k(\tau)} w$ denote a truncation of $w$: 
\[T_{k(\tau)} w(x) = \left\{ \begin{array}{ll} k & \text{if } w(x) > k, \\ -k & \text{if } w(x) < -k, \\ w(x) & \text{otherwise.} \end{array}\right. \]
We take $w^\tau = (T_{k(\tau)} w, \gamma\, T_{k(\tau)} w)$. Then,
\[ L_\tau w^\tau = T_{k(\tau)} w \to w \text{ in } X_0\] 
and
\[\|w^\tau\|_{X_\tau}^2 = \|T_{k(\tau)} w\|_{L^2(\Omega)}^2 + \tau \|\gamma \,T_{k(\tau)} w\|_{L^2(\partial \Omega)}^2 \to \|w\|_{L^2(\Omega)}^2,\]
provided that $k(\tau)$ converges to $\infty$ slowly enough so that the second term vanishes in the limit. Finally, 
\[\cE_\tau(w^\tau) = \int_{\Omega} |\nabla T_{k(\tau)} w| + \int_{\partial \Omega} |\gamma \,T_{k(\tau)} w - \gamma \,T_{k(\tau)} w|\to \int_{\Omega} |\nabla w| = \cE_0(w). \]
Now, we invoke Theorem \ref{thm:Mosco} to complete the proof.
\end{proof} 

\subsection{Dirichlet boundary condition as a limit of dynamic b.\,c.} 

Let $X_\tau$, $\cE_\tau$, $S_\tau$ be as in the previous subsection. We recall that the usual total variation flow with zero Dirichlet boundary condition $S_\infty$ is the gradient flow of $\cE_\infty \colon X_\infty \to [0, \infty]$ given by 
\[\cE_\infty(w)=\int_\Omega |Dw| + \int_{\partial \Omega} |\gamma w|,\]
where $X_\infty = L^2(\Omega)$. 

Let $L_\tau \colon X_\tau \to X_\infty$ be given by 
\[L_\tau (w^\tau_1, w^\tau_2) = w^\tau_1 . \]

\begin{theorem} 
$S_\tau$ uniformly 
converges to $S_\infty$ along $L_\tau$ as $\tau \to \infty$. 
\end{theorem} 

\begin{proof} 
Mosco convergence of $X_\tau$ as $\tau \to \infty$ was shown in the proof of Theorem \ref{thm:pflow_tau_convD}. We move on to checking the Mosco-convergence of $\cE_\tau$ to $\cE_\infty$ along $L_\tau$. We first prove \ref{H3}. Without loss of generality we can assume that $\cE_\tau(w^\tau) < \infty$. Then 
\begin{equation} \label{H3_Dir} \cE_\tau(w^\tau) = \int_\Omega |\nabla w^\tau_1| + \int_{\partial \Omega} |\gamma \, w^\tau_1 - w^\tau_2| \geq \int_\Omega |\nabla w^\tau_1| + \int_{\partial \Omega} |\gamma \, w^\tau_1| - \int_{\partial \Omega} |w^\tau_2| = \cE_\infty(w^\tau_1) - \int_{\partial \Omega} |w^\tau_2|. 
\end{equation} 
We note that the condition $\limsup_{\tau \to \infty} \|w^\tau\|_{X_\tau} < \infty$ implies $w^\tau_2 \to 0$ as $\tau \to \infty$ in $L^2(\partial \Omega)$. Using also the condition $L_\tau w^\tau = w^\tau_1 \rightharpoonup w$ together with weak lower semicontinuity of $\cE_\infty$ (which follows from strong lower semicontinuity and convexity), we easily deduce \ref{H3} from \eqref{H3_Dir}.
 
It remains to check \ref{H4}. Let $w \in X_\infty$ be such that $\cE_\infty(w) < \infty$. We take $w^\tau = (w, 0)$ as the recovery sequence. We have 
\[\cE_\tau(w^\tau) = \int_{\Omega} |\nabla  w| + \int_{\partial \Omega} |\gamma \,w|= \cE_\infty(w) \to \cE_\infty(w). \]
We invoke Theorem \ref{thm:Mosco} to finish the proof.
\end{proof} 

\section{Continuum limit of graph diffusion equations} \label{SgraphPDE}

Here we apply our abstract theory to a discrete-to-continuum limit. We only work with the subsequence of the ordered family $(\eps)$ such that $1/\varepsilon$ is an integer. With this understanding, all the statements of our abstract theory remain valid as written. 
We consider a graph $G_\varepsilon$ arising as a discretization of the flat torus $\mathbb{T}^n:=(\mathbb{R}/\mathbb{Z})^n \simeq [0,1]^n$ with mesh size $\varepsilon$.
That is, $G_\varepsilon=(V_\varepsilon,E_\varepsilon)$ is of the form
\begin{gather*}
    V_\varepsilon := \left(\varepsilon \mathbb{Z}/\mathbb{Z}\right)^n \simeq \{ 0, \varepsilon, 2\varepsilon, \ldots, 1-\varepsilon \}^n 
    \\ 
    E_\varepsilon := \left\{ \{z,\overline{z}\} \subset V_\varepsilon \bigm| \exists_i\ \forall_j\  
    \overline{z}_j - z_j = \pm \varepsilon \delta_{ij} \text{ in } \varepsilon \mathbb{Z}/\mathbb{Z} \right\} 
\end{gather*}
where $z=(z_1,\ldots,z_n)$. In other words, $E_\varepsilon$ is the set of (unordered) pairs of closest neighbors in $V_\epsilon$. 
 The set $V_\varepsilon$ is called the node set of $G_\varepsilon$ while the set $E_\varepsilon$ is called the edge set of $G_\varepsilon$.
 If $\{z,\overline{z}\}\in E_\varepsilon$, we simply write $z\sim\overline{z}$.

We define the $L^2$ inner product of functions $v_1,v_2:V_\varepsilon\to\mathbb{R}$ as
\[
    ( v_1,v_2 )_\varepsilon
    := \sum_{z\in V_\varepsilon} \varepsilon^n v_1(z) v_2(z) .
\]
This gives a Hilbert space structure to the set $X_\varepsilon$ of real functions on $V_\varepsilon$.
 This space is of finite dimension and as a set it coincides with $\mathbb{R}^{1/\eps^n}$.
For a fixed $1\leq p<\infty$, we consider a discrete $p$-Dirichlet energy $\mathcal{E}_{\eps} \colon X_\eps \to [0, \infty[$ given by 
\[
    \mathcal{E}_{\eps}(v) := \frac{1}{p} \sum_{\{z, \overline{z}\} \in E_\varepsilon} \varepsilon^n 
    \left( \left|v(z) - v(\overline{z}) \right|/\varepsilon \right)^p 
\]
and its gradient flow $S_\varepsilon$ with respect to the inner product $(\ ,\ )_\varepsilon$. Since in our considerations we treat $p$ as a fixed parameter, we suppress $p$ in our notation.

As an expected continuum limit of $X_\eps$, we set $X_0:=L^2(\mathbb{T}^n)$ equipped with the inner product
\[
    ( w_1,w_2 ) := \int_{\mathbb{T}^n} w_1(x) w_2(x), \quad
    w_1,w_2 \in L^2(\mathbb{T}^n).
\]
The expected Mosco-limit of $\mathcal{E}_\eps$ is the orthotropic $p$-Dirichlet energy $\mathcal{E}_\eps\colon X_0 \to [0, \infty]$ given by
\begin{equation*}
    \mathcal{E}_0(w) := \left \{
	\begin{array}{cl}
        \displaystyle \frac1p \int_{\mathbb{T}^n} |\nabla w|_{\ell^p}^p , & f\in W^{1,p}(\mathbb{T}^n), \\
        \infty, & w\in L^2(\mathbb{T}^n) \backslash W^{1,p} (\mathbb{T}^n) 
	\end{array}
	\right.
\end{equation*}
if $p>1$ and
\begin{equation*}
    \mathcal{E}_{0}(w) = \left \{
	\begin{array}{cl}
        \displaystyle  |D w|_{\ell^1} (\mathbb{T}^n), & w\in BV(\mathbb{T}^n) \\
        \infty, & w\in L^2(\mathbb{T}^n) \backslash BV (\mathbb{T}^n)
	\end{array}
	\right.
\end{equation*}
if $p =1$, where $|z|_{\ell^p}^p:=\sum_{i=1}^n |z_i|^p$.
Since $\mathcal{E}_{0}$ is a lower semicontinuous convex function in $X_0$, there is a unique gradient flow $u_t\in-\partial\mathcal{E}_{0}(u)$ with initial data $u_0\in L^2(\mathbb{T}^n)=\overline{D(\mathcal{E}_{0})}$.
 Let $S_0$ denote this gradient flow. We are interested in the convergence $S_\varepsilon$ to $S_0$ as $\varepsilon\downarrow0$.
 To apply our theory, we need to introduce suitable connecting operators  $L_\varepsilon$. This requires some preparation. 
%

We consider a standard simplicial decomposition of the unit cube.
 Let $\operatorname{Sym}(n)$ be the symmetric group of degree $n$, i.e., the group of permutations of $n$ elements.
 For $\sigma \in \operatorname{Sym}(n)$, we define
\[
    \Sigma_\sigma := \left\{ x=(x_1,\ldots,x_n) \bigm| 0 \leq x_{\sigma(1)} \leq x_{\sigma(2)} \leq \cdots \leq x_{\sigma(n)} \leq 1 \right\}.
\]
By definition, the unit cube $[0,1]^n$ is decomposed as
\[
    [0,1]^n = \bigcup_{\sigma\in\operatorname{Sym}(n)} \Sigma_\sigma.
\]
Since the Lebesgue measure $|\Sigma_\sigma\cap\Sigma_{\sigma'}|=0$ for $\sigma\neq\sigma'$, this decomposition is measure theoretically disjoint.
 Thus
\[
    |\Sigma_\sigma| = \frac{1}{n!},
\]
because the order (the number of elements) of $\operatorname{Sym}(n)$ equals $n!$.
 Each simplex has $n+1$ vertices.
 For example,
\[
    \Sigma_{\mathrm{id}} = \left\{ 0 \leq x_1 \leq x_2 \leq \cdots \leq x_n \leq 1 \right\}
\]
has vertices $\alpha_0 :=(0,\ldots,0)$, $\alpha_1 :=(0,\ldots,0,1)$, $\ldots$, $\alpha_{n-1} :=(0,1,\ldots,1)$, $\alpha_n :=(1,\ldots,1)$. In general, for $\sigma\in\operatorname{Sym}(n)$, $x\in\mathbb{R}^n$, we set $x^\sigma=(x_{\sigma(1)},\ldots,x_{\sigma(n)})$. Then, the vertices of $\Sigma_\sigma$ can be written as $\alpha_0^\sigma$, $\alpha_1^\sigma$, $\ldots$, $\alpha_{n-1}^\sigma$, $\alpha_n^\sigma$. For $k=1, \ldots, n$, $\alpha_k^\sigma-\alpha_{k-1}^\sigma$ equals $e_{\sigma(k)}$, the $\sigma(k)$-th vector of the standard basis of $\mathbb{R}^n$. In particular the vertices are affinely independent and so they form an affine basis. Thus, for any real-valued function $v$ defined on the set of all vertices of $\Sigma_\sigma$, there exists a unique affine function $f$ on $\Sigma_\sigma$ satisfying 
 \[f(z) = v(z)\]
 at each vertex $z\in\Sigma_\sigma$. We call $f$ the \textit{(affine) interpolation} of $v$ on $\Sigma_\sigma$. 

Next we introduce a decomposition of $\mathbb{T}^n$ into simplices at length scale $\varepsilon$: $\mathbb{T}^n = \bigcup Simp_\varepsilon$, where 
\[ Simp_\varepsilon = \{ z + \varepsilon \Sigma_\sigma \bigm| z \in V_\varepsilon, \ \sigma \in \operatorname{Sym}(n)\}.\]
The $n+1$ vertices of each simplex $\triangle \in Simp_\varepsilon$ belong to $V_\varepsilon$. Thus, given $w^\varepsilon \in X_\varepsilon$ and $\triangle \in Simp_\eps$, we can define $w_\triangle^\varepsilon$ to be the affine interpolation of $w^\varepsilon$ on $\triangle$. Then we define linear operators $L_\varepsilon \colon X_\varepsilon \to X_0$ by setting 
 \[L_\varepsilon w^\varepsilon(x) = w_\triangle^\varepsilon(x) \quad \text{for } x \in \triangle.\]
In order to prove Mosco-convergence of $X_\varepsilon$ to $X_0$ and $\mathcal{E}_\varepsilon$ to $\mathcal{E}_0$ along $L_\varepsilon$, we will need to compute the number of elements of the sets $\left\{ \triangle \in Simp_\varepsilon \bigm| z \in \triangle \right\}$ for a given $z \in V_\varepsilon$ and  $\left\{ \triangle \in Simp_\varepsilon \bigm| \{z, \overline{z}\} \subset \triangle \right\}$ for a given $\{z, \overline{z}\} \in E_\varepsilon$. We note that the number of simplices in the simplicial decomposition of a cube that contains a given vertex depends on the vertex. For example, in the case of the unit square, vertex $(0,0)$ is contained in two triangles, while vertex $(0,1)$ only in one. Nevertheless, translational symmetry of $G_\varepsilon$ allows us to obtain the value of $\# \left\{ \triangle \in Simp_\varepsilon \bigm| z \in \triangle \right\}$ without computing $\# \left\{ \Sigma_\sigma \bigm| y \in \Sigma_\sigma \right\}$ for each vertex $y$ of $[0,1]^n$.
 
\begin{lemma}\label{lem:counting} 
 For any node $z \in V_\varepsilon$,
 \[\# \left\{ \triangle \in Simp_\varepsilon \bigm| z \in \triangle \right\} = (n+1)!.\]
For any edge $\{z, \overline{z}\} \in E_\varepsilon$,
 \[\# \left\{ \triangle \in Simp_\varepsilon \bigm| \{z, \overline{z}\} \subset \triangle \right\} = n!.\]
\end{lemma} 

\begin{proof} 
We observe that $\# V_\varepsilon = \varepsilon^{-n}$ and $\# Simp_\varepsilon = \varepsilon^{-n} \,n!$.  As we have already mentioned, each simplex in $Simp_\varepsilon$ has exactly $n+1$ vertices which all belong to the node set $V_\varepsilon$. By translational invariance of $G_\varepsilon$, the number of simplices that contain a given node does not depend on the node. Thus, by counting in two ways the number of elements of the set $\{ (z, \triangle)\in V_\varepsilon\times Simp_\varepsilon \,|\, z \in \triangle\}$, we get for any $z \in V_\varepsilon$,  
\[ \# \left\{ \triangle \in Simp_\varepsilon \bigm| z \in \triangle \right\} = \frac{(n+1) \, \# Simp_\varepsilon}{\# V_\varepsilon} = (n+1)!.\] 

Next, we have $\# E_\varepsilon = \varepsilon^{-n}\, n$. On the other hand, each simplex has $n$ edges that belong to $E_\varepsilon$. By invariance of $G_\varepsilon$ with respect to translations and relabeling axes, the number of simplices that contain a given edge does not depend on the edge. Thus, by counting in two ways the number of elements of the set $\{ (\{z,\overline{z}\}, \triangle)\in E_\varepsilon\times Simp_\varepsilon \,|\, \{z,\overline{z}\} \subset \triangle\}$, we get for any $\{z,\overline{z}\}\in E_\varepsilon$, 
\[ \# \left\{ \triangle \in Simp_\varepsilon \bigm| \{z,\overline{z}\} \subset \triangle \right\} = \frac{n \, \# Simp_\varepsilon}{\# E_\varepsilon} = n!.\qedhere\] 
\end{proof} 

\begin{remark} \label{Rnur}
The proof of Lemma \ref{lem:counting} given above avoids counting the number $\#\{\Sigma_\sigma\mid y\in\Sigma_\sigma\}$ for each vertex $y$ of $[0,1]^n$ and $\#\{\Sigma_\sigma\bigm| \{y,\bar{y}\}\in\Sigma_\sigma\}$ for each edge $\{y,\bar{y}\}$ of $[0,1]^n$. Here we provide an alternative proof, where we explictly compute those numbers, which depend on the vertex $y$ or edge $\{y,\bar{y}\}$ (compare Figure \ref{F2d}, \ref{F3d}). 
\begin{figure}[h]
  \begin{minipage}[b]{0.5\linewidth}
\centering
\includegraphics[keepaspectratio, scale=0.2]{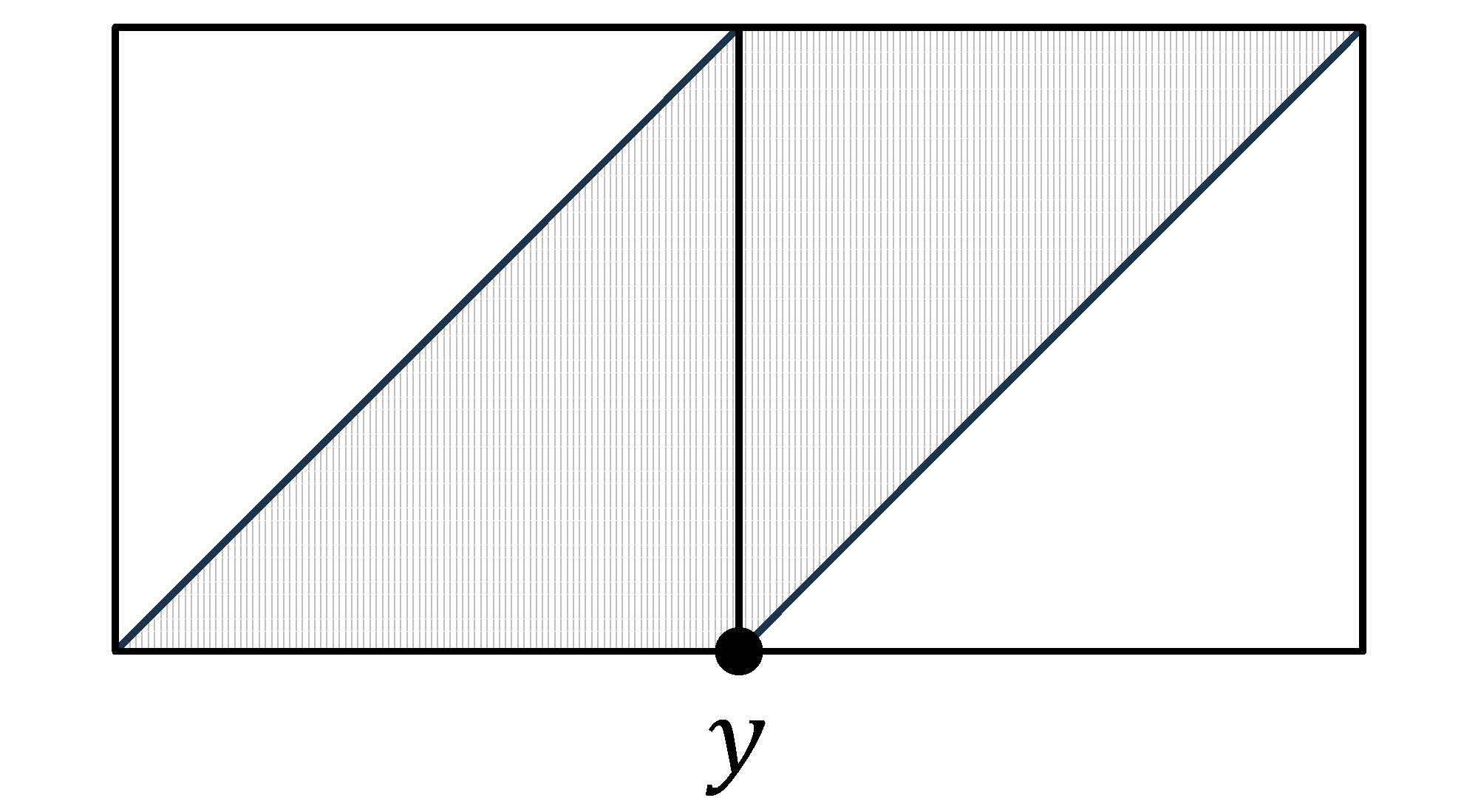} 
\caption{$2$-dimensional case\label{F2d}}
	\end{minipage}
  \begin{minipage}[b]{0.5\linewidth}
\centering 
\includegraphics[keepaspectratio, scale=0.27]{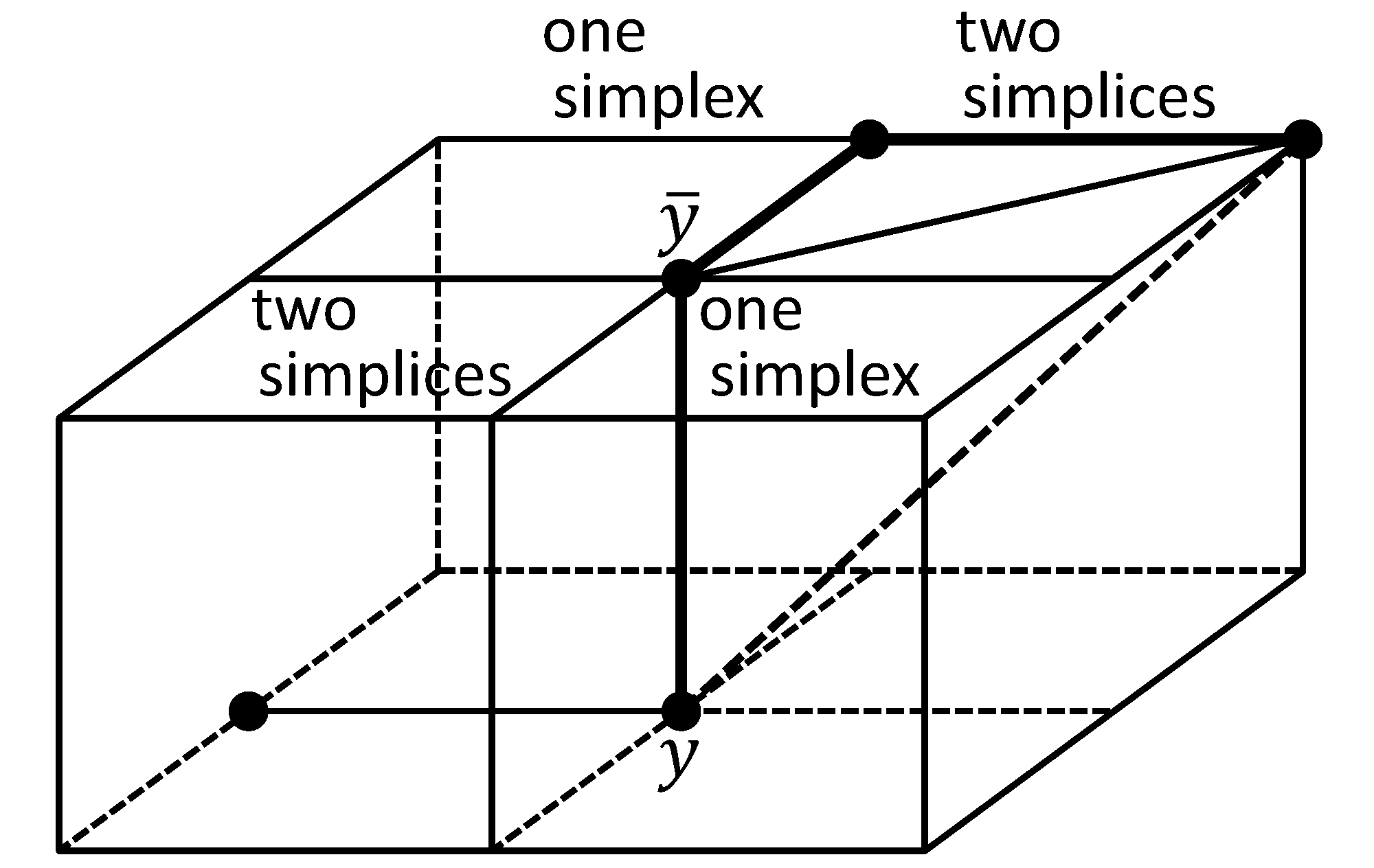}
\caption{$3$-dimensional case\label{F3d}}
  \end{minipage}
\end{figure}. 

The set of vertices of $[0,1]^n$ coincides with $\{0,1\}^n$. Let $k$ be the number of $1$'s among the coordinates of $y \in \{0,1\}^n$, i.e., $k = \sum y_j$. Then
\begin{equation} \label{EVern}
    \#\{\Sigma_\sigma\mid y\in\Sigma_\sigma\}
    = (n-k)! k!.
\end{equation}
Indeed, we may assume that $y=\alpha_k$.
A simplex $\Sigma_\sigma$ containing $\alpha_k$ is of the form
\[
    \Sigma_{\sigma_1,\sigma_2}
    = \left\{ 0 \leq x_{\sigma_1(1)}
    \leq \cdots \leq x_{\sigma_1(n-k)}
    \leq x_{\sigma_2(n-k+1)} \leq \cdots
    \leq x_{\sigma_2(n)} \leq 1 \right\},
\]
where $\sigma_1\in\operatorname{Sym}(n-k)$ and $\sigma_2\in\operatorname{Sym}(k)$.
 The total number of such simplices equals the order of $\operatorname{Sym}(n-k)$ times that of $\operatorname{Sym}(k)$, i.e., $(n-k)!k!$.
Now, there are exactly $_n C_k$ vertices $y \in \{0,1\}^n$ such that $\sum y_j =k$. Thus, by \eqref{EVern}, 

\begin{multline*}
   \sum_{y \in \{0,1\}^n} \#\{\Sigma_\sigma
    \mid y \in \Sigma_\sigma\}  
    = \sum_{y \in \{0,1\}^n} \sum_{k=0}^n \#\left\{\Sigma_\sigma \ \bigg| \ \sum y_j = k \text{ and } y \in \Sigma_\sigma \right\}\\
    = \sum_{k=0}^n {}_n C_k (n-k)!k!
    = (n+1)n! = (n+1)!.
\end{multline*}
Finally, we observe that every node $z \in V_{\varepsilon}$ is a vertex of exactly $2^n$ elementary lattice hypercubes: for each $y \in \{0,1\}^n$, the node $z$ is in position $y$ with respect to exactly one of them. Thus, we have 
\[ \#\{\triangle \in Simp_\varepsilon 
    \mid z\in\triangle \} = \sum_{y \in \{0,1\}^n} \#\{\Sigma_\sigma
    \mid y \in \Sigma_\sigma\} = (n+1)!,\]
retrieving the first part of Lemma \ref{lem:counting}. 

The case of edges is a bit more involved.
 Let $\{y,\bar{y}\}$ be an edge of $[0,1]^n$.
 We can assume that $\{y,\bar{y}\}$ is \emph{vertical}, i.e., the last component of $y$ equals $0$ while the last component of $\bar{y}$ equals $1$, i.e., $y=(y',0)$, $\bar{y}=(y',1)$.
 Let $k$ be the number of $1$ among the coordinates of $y'$, i.e., $k = \sum y'_j$.
 Then
\begin{equation} \label{EEdge}
    \#\left\{\Sigma_\sigma \bigm|
    \{y,\bar{y}\} \in \Sigma_\sigma \right\}
    = (n-k-1)! k!.
\end{equation}
Indeed, owing to symmetry with respect to relabeling the axes, we may assume that $y'=\alpha'_k$, where $\alpha'_k=(\overbrace{0,\ldots,0}^{n-1-k},\overbrace{1,\ldots,1}^{k})$.

Simplex $\Sigma_\sigma$ contains the edge $\left\{(\alpha'_k,0),(\alpha'_k,1)\right\}$ if and only if it is of the form
\[
    \Sigma_{\sigma_1, \sigma_2} = \{ 0 \leq x_{\sigma_1(1)} \leq \cdots \leq x_{\sigma_1(n-k-1)}
    \leq x_n \leq x_{\sigma_2(n-k)} \leq
    \cdots \leq x_{\sigma_2(n-1)} \leq 1 \},
\]
where $\sigma_1\in\operatorname{Sym}(n-k-1)$ and $\sigma_2\in\operatorname{Sym}(k)$.
 The total number of such simplices equals the order of $\operatorname{Sym}(n-k-1)$ times that of $\operatorname{Sym}(k)$, i.e., $(n-k-1)!k!$.
 This yields \eqref{EEdge}.
Now, there are $_{n-1}C_k$ ways to choose which $k$ coordinates of $y'$ are equal to $1$. Thus, by \eqref{EEdge}, 
\begin{multline*}
    \sum_{y' \in \{0,1\}^{n-1}} \hspace{-10pt} \# \left\{ \Sigma_\sigma \bigm|
    \{(y',0), (y',1)\} \in \Sigma_\sigma \right\} 
    \\ = \sum_{y' \in \{0,1\}^{n-1}} \sum_{k=0}^{n-1} \# \left\{ \Sigma_\sigma \bigm| \{(y',0), (y',1)\} \in \Sigma_\sigma ,\  \sum y'\ =k \right\} \\    
    = \sum_{k=0}^{n-1} {}_{n-1}C_k (n-k-1)!k!
    = \sum_{k=0}^{n-1} (n-1)! = n!.
\end{multline*}
Any edge $\{z, \bar{z}\} \in E_\varepsilon$ belongs to exactly $2^{n-1}$ elementary lattice hypercubes. Assuming that $z_n \neq \bar{z}_n$, for each vertical edge $\{y, \bar{y}\}$ of $[0,1]^n$, the edge $\{z, \bar{z}\}$ is in position $\{y, \bar{y}\}$ with respect to exactly one of them. Thus, we obtain  
\[\# \left\{ \triangle \in Simp_\varepsilon \bigm|
    \{ z,\bar{z} \} \in \triangle \right\} = \sum_{y' \in \{0,1\}^{n-1}} \hspace{-10pt} \# \left\{ \Sigma_\sigma \bigm|
    \{(y',0), (y',1)\} \in \Sigma_\sigma \right\} = n!\] 
 which concludes the alternative proof of Lemma \ref{lem:counting}. 
\end{remark}
%


\begin{lemma}\label{lem:graph_X_conv}
    $X_\varepsilon$ Mosco-converge to $X_0$ along $L_\varepsilon$. 
\end{lemma}

\begin{proof} 
Let $w^\eps \in X_\varepsilon$. Observe that the vertices of any $\triangle \in Simp_\varepsilon$ are its extreme points. In particular, $\triangle$ is the convex hull of its vertices, i.e. 
\[\triangle = \left\{\sum_{j=0}^n \lambda_j z_j \biggm| 0 \leq \lambda_j \leq 1, \  \sum_{j=0}^n \lambda_j = 1\right\}.\]
Moreover, since  the vertices form an affine basis, each point $x \in \triangle$ has a unique expression in barycentric coordinates $\lambda_0, \ldots, \lambda_n$. Thus, the affine interpolation $w_\triangle^\varepsilon$ can be explicitly written as 
\[w_\triangle^\varepsilon(x) = \sum_{j=0}^n \lambda_j w^\varepsilon(z_j) \quad \text{for } x = \sum_{j=0}^n \lambda_j z_j \in \triangle.\] 
We have 
\[ \|L_\varepsilon w^\varepsilon \|_{X_0}^2 = \sum_{\triangle \in Simp_\varepsilon} \int_\triangle |w_\triangle^\varepsilon|^2. \]
For a given $\triangle$, by Jensen's inequality
\[ \int_\triangle |w_\triangle^\varepsilon|^2 = \int_\triangle \left|\sum_{j=0}^n \lambda_j(x) w^\varepsilon(z_j)\right|^2 dx \leq \int_\triangle \sum_{j=0}^n \lambda_j(x) \left|w^\varepsilon(z_j)\right|^2 dx = \sum_{j=0}^n \left|w^\varepsilon(z_j)\right|^2 \int_\triangle \lambda_j(x)\,dx.\]
It is not difficult to show, by induction over the space dimension $n$, that
\[\int_\triangle \lambda_j(x)\,dx = \frac{1}{\varepsilon^n \, (n+1)!}.\]
independently of $j$. This is a special case of a general formula for the integral of a monomial in barycentric coordinates over a simplex given in \cite[Section 2]{Chen2008}. (An explicitly written out proof can be found in \cite{VermolenSegal}.) Thus, by Lemma \ref{lem:counting},
\[\|L_\varepsilon w^\varepsilon \|_{X_0}^2 \leq \frac{1}{\varepsilon^n \, (n+1)!}\sum_{\triangle \in Simp_\varepsilon} \sum_{z \in \triangle \cap V_\varepsilon} \left|w^\varepsilon(z)\right|^2 = \varepsilon^{-n} \sum_{z \in V_\varepsilon} \left|w^\varepsilon(z)\right|^2 = \|w^\varepsilon\|_{X_\varepsilon}^2, \]
whence \ref{H1} holds. 

For a given $w \in X_0$, let $w^\delta := \varrho_\delta * w$, where $\varrho_\delta$ is a standard mollifying kernel. We have $w^\delta \to w$ in $X_0$ as $\delta \to 0^+$. For a fixed $\delta>0$ and any $\varepsilon \in 1/\mathbb{N}$, let $\widetilde{w}^{\delta, \varepsilon} \in X_\varepsilon$ be given by 
\[ \widetilde{w}^{\delta, \varepsilon}(z) = w^\delta(z) \quad \text{for } z \in V_\varepsilon. \]
Since $w^\delta$ are continuous (and therefore Riemann integrable), we see that 
\[ \|\widetilde{w}^{\delta, \varepsilon}\|_{X_\varepsilon}^2 = \sum_{z \in V_\varepsilon} \varepsilon^{-n} |\widetilde{w}^{\delta, \varepsilon}(z)|^2 = \sum_{z \in V_\varepsilon} \varepsilon^{-n} |w^\delta(z)|^2 \to \int_{\mathbb{T}^n} |w^\delta|^2 = \|w^{\delta}\|_{X_0}^2\]
and 
\[L_\varepsilon \widetilde{w}^{\delta, \varepsilon} \to w^\delta \quad \text{in } X_0\]
as $\varepsilon \to 0^+$. By the usual diagonal procedure, choosing a sequence $\delta(\varepsilon)$ that converges to $0$ sufficiently slowly as $\varepsilon \to 0^+$, we obtain a sequence $w^\varepsilon := \widetilde{w}^{\delta(\varepsilon), \varepsilon}$ that satisfies \ref{H2}. 
\end{proof} 

Now it remains to show that $\mathcal{E}_\varepsilon$ Mosco-converge to $\mathcal{E}_0$ along $L_\varepsilon$. The crucial ingredient is the following observation. 

\begin{lemma}\label{lem:graph_E_eq}
    For any $\varepsilon \in 1/\mathbb{N}$ and $w^\varepsilon \in X_\varepsilon$,
    \[ \mathcal{E}_\varepsilon(w^\varepsilon) = \mathcal{E}_0(L_\varepsilon w^\varepsilon).\] 
\end{lemma}

\begin{proof} 
We recall that the vertices of the simplex $\triangle = z + \varepsilon \Sigma_\sigma \in Simp_\varepsilon$ can be written as $z + \varepsilon \alpha_k^\sigma$, $k=0, \ldots, n$. Moreover, 
\[z + \varepsilon \alpha_k^\sigma - (z + \varepsilon \alpha_{k-1}^\sigma) = \varepsilon e_{\sigma(k)}\]
for $k=1, \ldots, n$. Thus, given $w^\varepsilon \in X_\varepsilon$, the affine interpolation $w^\varepsilon_\triangle$ can be expressed in the form
\[ w^\varepsilon_\triangle((x_1, \ldots, x_n)) = w^\varepsilon(z) + \sum_{k=1}^n \frac{x_{\sigma(k)}- z_{\sigma(k)}}{\varepsilon} \left( w^\varepsilon(z + \varepsilon \alpha_k^\sigma) - w^\varepsilon(z + \varepsilon \alpha_{k-1}^\sigma)\right)\]
and we have 
\[ \frac{\partial w^\varepsilon_\triangle}{\partial x_{\sigma(k)}} = \frac{w^\varepsilon(z + \varepsilon \alpha_k^\sigma) - w^\varepsilon(z + \varepsilon \alpha_{k-1}^\sigma)}{\varepsilon} .\]
Therefore, we can write 
\[ \int_\triangle \sum_{k=1}^n \left|\frac{\partial w^\varepsilon_\triangle}{\partial x_k}\right|^p =  \frac{\varepsilon^{-n}}{n!}\sum_{k=1}^n \left|\frac{\partial w^\varepsilon_\triangle}{\partial x_{\sigma(k)}}\right|^p = \frac{\varepsilon^{-n}}{n!} \sum_{\substack{\{z, \overline{z}\} \in E_\varepsilon\\ \{z, \overline{z}\} \subset \triangle}} \left| \frac{w^\varepsilon(\overline{z}) - w^\varepsilon(z)}{\varepsilon} \right|^p  \]
and 
\[ \mathcal{E}_0(L_\varepsilon w^\varepsilon) = \sum_{\triangle \in Simp_\varepsilon} \int_\triangle \sum_{k=1}^n \left|\frac{\partial w^\varepsilon_\triangle}{\partial x_k}\right|^p = \sum_{\triangle \in Simp_\varepsilon} \frac{\varepsilon^{-n}}{n!} \sum_{\substack{\{z, \overline{z}\} \in E_\varepsilon\\ \{z, \overline{z}\} \subset \triangle}} \left| \frac{w^\varepsilon(\overline{z}) - w^\varepsilon(z)}{\varepsilon}\right|^p.\]
Eventually, appealing to Lemma \ref{lem:counting}, we conclude that 
\[ \mathcal{E}_0(L_\varepsilon w^\varepsilon) =  \varepsilon^{-n} \sum_{\{z, \overline{z}\} \in E_\varepsilon} \left| \frac{w^\varepsilon(\overline{z}) - w^\varepsilon(z)}{\varepsilon}\right|^p = \mathcal{E}_\varepsilon(w^\varepsilon).\qedhere\]
\end{proof} 

\begin{lemma}  \label{lem:graph_E_conv}
    $\mathcal{E}_\varepsilon$ Mosco-converge to $\mathcal{E}_0$ along $L_\varepsilon$. 
\end{lemma} 

\begin{proof} 
 Suppose that $w \in X_0$, $w^\varepsilon \in X_\varepsilon$, $ 1/\varepsilon \in\mathbb{N}$ are such that $L_\varepsilon w^\varepsilon \rightharpoonup w$ in $X_0$. Then, by Lemma \ref{lem:graph_E_eq}, 
 \[ \liminf_{\eps \to 0^+} \mathcal{E}_\varepsilon(w^\varepsilon) = \liminf_{\eps \to 0^+} \mathcal{E}_0(L_\varepsilon w^\varepsilon) \geq \mathcal{E}_0(w), \]
because $\mathcal{E}_0$ is convex and lower semicontinuous. Thus, \ref{H3} holds. 

It remains to show \ref{H4}. As in the proof of Lemma \ref{lem:graph_X_conv}, we fix $\delta>0$, $1/\varepsilon \in \mathbb{N}$, $w \in X_0$. 
We may assume that $\mathcal{E}_0(w) < \infty$, whence $w \in W^{1,p}(\mathbb{T}^n)$ if $p>1$ and $w \in BV(\mathbb{T}^n)$ if $p=1$. For $p>1$ we let $w^\delta := \varrho_\delta * w$, where $\varrho_\delta$ is a standard mollifying kernel. For $p=1$ w take t $w^\delta=$ to be the sequence of smooth functions converging to $w$ in the strict sense, see \cite[Remark 3.22]{AFP}. 
In any case, we have $\mathcal{E}_0(w^\delta)\to \mathcal{E}_0(w)$ as $\delta \to 0^+$. We also take $\widetilde{w}^{\delta, \varepsilon} \in X_\varepsilon$ given by 
\[ \widetilde{w}^{\delta, \varepsilon}(z) = w^\delta(z) \quad \text{for } z \in V_\varepsilon. \]
We have already observed that
\[ \|\widetilde{w}^{\delta, \varepsilon}\|_{X_\varepsilon}^2  \to \|w^{\delta}\|_{X_0}^2 \quad \text{and} \quad  L_\varepsilon \widetilde{w}^{\delta, \varepsilon} \to w^\delta \quad \text{in } X_0\]
as $\varepsilon \to 0^+$. Since $w^\delta$ are continuously differentiable, we also have 
\[L_\varepsilon \widetilde{w}^{\delta, \varepsilon} \to w^\delta \quad \text{in } W^{1,p}(\mathbb{T}^n)\]
whence, by Lemma \ref{lem:graph_E_eq},
\[ \mathcal{E}_\varepsilon (\widetilde{w}^{\delta, \varepsilon}) = \mathcal{E}_0 (L_\varepsilon \widetilde{w}^{\delta, \varepsilon}) \to \mathcal{E}_0 (w^\delta) \quad \text{as } \varepsilon \to 0^+.\]
Again, choosing a sequence $\delta(\varepsilon)$ that converges to $0$ sufficiently slowly as $\varepsilon \to 0^+$, we obtain a sequence $w^\varepsilon := \widetilde{w}^{\delta(\varepsilon), \varepsilon}$ that satisfies \eqref{H2_conv} and $\mathcal{E}_\varepsilon (w^\varepsilon)  \to \mathcal{E}_0 (w)$ as $\varepsilon \to 0^+$. 
\end{proof} 

From Lemmata \ref{lem:graph_X_conv} and \ref{lem:graph_E_conv}, using Theorem \ref{thm:Mosco}, we immediately deduce:
\begin{theorem}
Gradient flows $S_\varepsilon$ uniformly converge to $S_0$ along $L_\varepsilon$. \qed
\end{theorem}

\section*{Acknowledgments}
The work of the first author was partly supported by the Japan Society for the Promotion of Science (JSPS) through grants KAKENHI Grant Numbers 19H00639, 20K20342, 24K00531 and 24H00183 and by Arithmer Inc., Daikin Industries, Ltd.\ and Ebara Corporation through collaborative grants.

The work of the second author was partly supported by the grant of the National Science Centre (NCN), Poland no.\ 2024/55/D/ST1/03055. Part of this work was created during the second author’s JSPS Postdoctoral Fellowship at the University
of Tokyo.

The third author enjoyed support of the IDUB program of the University of Warsaw, which enabled his visits to collaborators.

\end{document}